\documentclass[pdflatex,sn-mathphys-num]{sn-jnl}

\usepackage{float}
\usepackage{geometry}
\usepackage{graphicx}%
\usepackage{multirow}%
\usepackage{amsmath,amssymb,amsfonts}%
\usepackage{amsthm}%
\usepackage{mathrsfs}%
\usepackage[title]{appendix}%
\usepackage{xcolor}%
\usepackage{textcomp}%
\usepackage{manyfoot}%
\usepackage{booktabs}%
\usepackage{algorithm}%
\usepackage{algorithmicx}%
\usepackage{algpseudocode}%
\usepackage{listings}%
\usepackage{bbm}
\usepackage{bm}
\usepackage{enumerate}
\usepackage{setspace}

\DeclareMathOperator*{\argmin}{arg\,min}

\theoremstyle{thmstyleone}%
\newtheorem{theorem}{Theorem}
\newtheorem{proposition}[theorem]{Proposition}%

\newtheorem{lemma}[theorem]{Lemma}

\theoremstyle{thmstyletwo}%
\newtheorem{remark}{Remark}%
\newtheorem{definition}{Definition}%

\theoremstyle{thmstylethree}%

\begin{document}

\title[Article Title]{Efficient Function Approximation Under Heteroskedastic Noise}


\author[1]{\fnm{Yuji} \sur{Nakatsukasa}}\email{yuji.nakatsukasa@maths.ox.ac.uk}

\author*[1]{\fnm{Yifu} \sur{Zhang}}\email{yifu.zhang@univ.ox.ac.uk}

\affil*[1]{\orgdiv{Mathematics Institute}, \orgname{University of Oxford}, \orgaddress{\city{Oxford}, \postcode{OX2 6CG},\country{UK}}}


\abstract{Approximating a function $f(x)$ on $[-1,1]$ based on $N+1$ samples is a classical problem in numerical analysis. If the samples come with heteroskedastic noise depending on $x$ of variance $\sigma(x)^2$, an $O(N\log N)$ algorithm for this problem has not yet been found in the current literature. In this paper, we propose a method called HeteroChebtrunc, adapted from an algorithm named NoisyChebtrunc. Using techniques in high-dimensional probability, we show that with high probability, HeteroChebtrunc achieves a tighter infinity-norm error bound than NoisyChebtrunc under heteroskedastic noise. This algorithm runs in $O(N+\hat{N}\log \hat{N})$ operations, where $\hat{N}\ll N$ is a chosen parameter. While investigating the properties of HeteroChebtrunc, we also derive a high-probability non-asymptotic relative error bound on the sample variance estimator for subgaussian variables, which is potentially another result of broader interest. We provide numerical experiments to demonstrate the improved uniform error of our algorithm.}



\maketitle

\section{Introduction}
\setcounter{page}{1}

A central theme in approximation theory is polynomial approximation of a function $f(x)$ on $[-1,1]$. One common approach is Lagrange interpolation, which first samples $f$ at $ \{x_i\}_{i=0}^N $, for some $N$ and $x_i$, then finds the unique degree $N$ polynomial $p_N$ such that $p_N(x_i) = f(x_i)$. When the function can be sampled precisely and the sample points chosen freely, the interpolation based on Chebyshev points, that is, $x_i = \cos(i\pi/N)$, is one of the most established methods. Chebyshev interpolation is known to be near-best and achieves a spectral rate of convergence \cite[Chap. 7, 8, 16]{ATAP}, e.g. exponential if $f$ is analytic and algebraic if $f$ is differentiable. However, when sampling comes with noise as is common in real-world applications, one would need to utilise a method that still guarantees fast convergence and hopefully reduces such noise. 

The setting of our problem is to approximate a function $f:[-1,1] \rightarrow \mathbb{R}$ based on sampling in $x \in [-1,1]$ given by

\[y = f(x) + \epsilon_x,\]
where $\epsilon_x$ follows some error distribution, potentially dependent on $x$. One typical example would be $\epsilon_x \sim N(0, \sigma(x)^2)$, the Gaussian distribution with mean $ 0 $ and variance $\sigma(x)^2$ for some $\sigma(x)\geq 0$. We assume mean of $\epsilon_x$ is 0, as otherwise the mean becomes part of $f(x)$. We only consider $[-1,1]$ as any interval $[a,b]$ can be transformed to $[-1,1]$ via a standard affine mapping. 

The ideal solution for this problem is to find $p^* = \argmin \| f - p \|_\infty$, the best polynomial approximation to $f$. However, in a noisy setting this is not possible, so instead our goal is to obtain $p_n$, a degree $n$ polynomial which reduces the uniform error $\|f-p_n\|_\infty$ as much as possible.

One powerful algorithm to obtain such an approximant is NoisyChebtrunc, proposed by Matsuda and Nakatsukasa \cite{yuji} in 2024. In short, given a computational budget of $N+1$ allowed samples, it first interpolates $f$ at $N+1$ Chebyshev points to yield $\hat{p}_N$, a degree $N$ interpolant of $f$, then truncates the Chebyshev extension of $\hat{p}_N$ to an appropriate degree $n$, where $n<N$ (usually $n\ll N$), using a statistical criterion\footnote{See the discussion following Algorithm \ref{alg:noisychebtrunc}}, Mallow's $C_p$, which produces the final approximant $p_n$. It has several advantages: first of all, it is very computationally efficient, requiring only $O(N\log N)$ operations. Secondly, it is numerically stable, as inherited from Chebyshev interpolation. Finally, it also has very attractive convergence properties: assuming the noise\footnote{Here one also assumes noise is subgaussian/subexponential, which is explained in later sections.} is i.i.d. with variance $\sigma^2$, it has a spectral rate of convergence until error reaches $O(\sigma\sqrt{\frac{n}{N}})$, then error continues to decay as $O(\frac{1}{\sqrt{N}})$. We give more details of NoisyChebtrunc in Section \ref{sec:chebandnoisy}.

In this paper, we propose an extension to NoisyChebtrunc, called HeteroChebtrunc, designed to tackle \emph{heteroskedastic} independent noise. Heteroskedasticity refers to the scenario where the noise variance is no longer uniform, which arises in many applications, e.g. environmental risk mapping, financial volatility forecasting, and others \cite{envi-risk,volatility-forecast}. NoisyChebtrunc is designed to handle i.i.d., and hence homoskedastic (uniform variance) noise, and under heteroskedastic assumptions, NoisyChebtrunc will have error proportional to the largest noise divided by $\sqrt{N}$ \cite{yuji}, which is not ideal. In HeteroChebtrunc, we interpolate at $\hat{N} \ll N$ Chebyshev points, and employ a weighted sampling scheme. This allows us to have an error which scales proportional to $\frac{\|\bm{\sigma}\|_2}{\sqrt{\hat{N}}}$ rather than $\|\bm{\sigma}\|_\infty$, where $\bm{\sigma} = (\sigma(x_i))_{i=0}^{\hat{N}}$. This can be a significant improvement if the noise is heteroskedastic as we demonstrate in our experiments.

In addition to the improved accuracy, HeteroChebtrunc also reduces the time complexity of NoisyChebtrunc from $O(N\log N)$ to $O(N+\hat{N}\log \hat{N})$, and a significant difference can be observed when conducting numerical experiments. For example, using the authors' standard laptop, running 100 trials of NoisyChebtrunc with $N = 10^6$ took 51.3407 seconds, but HeteroChebtrunc (with $\hat{N} = 10^3$) took only $1.7851$ seconds. 

Compared with existing methods, our algorithm has many desirable properties, some of the key advantages include:

\begin{enumerate}
	\item By utilising Chebyshev interpolation, HeteroChebtrunc inherits its computational efficiency and numerical stability, which cannot be guaranteed in other methods, e.g. if sample points are equispaced \cite{yuji}. In fact, our algorithm is faster than NoisyChebtrunc's $O(N\log N)$, although this is at the cost of an additional requirement: repeated sampling at the same $x$ yields independent evaluations of noise $\epsilon_x$.
	\item Compared with NoisyChebtrunc, not only does HeteroChebtrunc address heteroskedasticity and achieve smaller uniform error, as an estimator it also has a smaller variance than NoisyChebtrunc, as demonstrated in Figure \ref{fig:histogramhetero}. This means that it produces tighter confidence intervals.
	\item Compared with estimators from nonparametric regression, one common issue with statistical methods like local regression and penalisation spline is that they have non-uniform variance under heteroskedasticity, that is, they typically have larger error on noisier regions. We observe the same phenomenon with NoisyChebtrunc in Figure \ref{fig:f-pn}. However, by redistributing the noise evenly across sample points, our estimator significantly reduces this effect, allowing for accurate estimation even at $x$ where $\sigma(x)$ is large. This is also why we are able to achieve a better uniform error, as traditional methods will inevitably have a larger pointwise error at noisier $x$.
\end{enumerate}

\subsection{Motivation}

There are two key ideas behind HeteroChebtrunc, the first of which is \emph{weighted sampling}. In order to motivate this idea, let us assume that the function $\sigma(x)$ is known. We introduce a weighted sampling scheme in Section \ref{sec:noisy-hetero}, which instead of sampling at the $N+1$ Chebyshev points, we choose $\hat{N} \ll N$ and sample at the $\{x_i\}_{i=0}^{\hat{N}}$ Chebyshev points. This allows us to sample multiple times at each of the $x_i$, which is a simple but powerful idea as we later demonstrate. Specifically, we propose Algorithm \ref{alg:noisychebtrunchetero}, which addresses heteroskedasticity by sampling $k_i$ times at each of the $x_i$ Chebyshev points, with $k_i$ assigned based on noise levels at each node. Intuitively, one takes more samples, i.e. chooses larger $k_i$, at noisier nodes, and fewer at cleaner nodes, so that by taking $y_i$ as the average of the $k_i$ samples at $x_i$, the noise level becomes more uniform and the maximum noise is reduced. We then find the Chebyshev interpolant $\hat{p}_{\hat{N}}$ through $\{(x_i,y_i)\}_{i=0}^{\hat{N}}$ and truncate using Mallow's $C_p$ as in NoisyChebtrunc. Given the noise level at each node $ \bm{\sigma} = (\sigma_i)$, where $\sigma_i^2$ is the variance of noise at $x_i$, $i= 0,1,...,\hat{N}$, we show that Algorithm \ref{alg:noisychebtrunchetero} improve the uniform error $\|f-p_n\|_\infty$ of NoisyChebtrunc\footnote{That is to say, if an approximant from NoisyChebtrunc has uniform error $\|f-p_n\|_\infty$, then we expect HeteroChebtrunc to have error $\frac{\|\bm{\sigma}\|_2}{\|\bm{\sigma}\|_\infty \sqrt{\hat{N}}}\|f-p_n\|_\infty$.} by a factor of $\frac{\|\bm{\sigma}\|_2}{\|\bm{\sigma}\|_\infty \sqrt{\hat{N}}} \leq 1$ for large $N$. This algorithm inherits most of the attractive properties of NoisyChebtrunc, and improves its convergence under heteroskedastic noise. 

The second key idea is \emph{pre-sampling}, which we introduce in order to resolve Algorithm \ref{alg:noisychebtrunchetero}'s reliance on knowledge of $\sigma(x)$. To be specific, we allocate a portion $r$ of the sampling budget to pre-samples. That is, we first apply $rN$ samples\footnote{$rN$ assumed to be an integer.}, and allocate them uniformly at each of the $\hat{N}+1$ Chebyshev points $x_i$ to estimate their noise variances $\sigma_i^2$ using the sample variance estimator, then apply weighted sampling using the rest of $(1-r)N$ samples as in Algorithm \ref{alg:noisychebtrunchetero} using our estimations of $\sigma_i$. 

This gives an intuitive description of our main algorithm HeteroChebtrunc: choosing $\hat{N} +1 \ll N$ Chebyshev points and $0<r<1$, we 

\begin{enumerate}
	\item Pre-sample to estimate the noise variance $\sigma_i^2$ at each $x_i$ using the sample variance estimator;
	\item Apply weighted sampling using the remaining $(1-r)N$ samples, based on our estimation of $\sigma_i^2$ from the pre-sampling; we sample more at noisier nodes and less at cleaner nodes, and take $y_i$ as the average of the $k_i$ samples at $x_i$;
	\item Find the unique Chebyshev interpolant $\hat{p}_{\hat{N}}$ through $\{(x_i, y_i)\}_{i=0}^{\hat{N}}$;
	\item Truncate to degree $n$ using Mallow's $C_p$ to yield output $p_n$.
\end{enumerate}

We give the specific details of HeteroChebtrunc and those of Algorithm \ref{alg:noisychebtrunchetero} in Section \ref{sec:noisy-unknown} and Section \ref{sec:noisy-hetero} respectively. We summarise the properties of these three algorithms in Table \ref{tab:compare} below for direct comparison:

\begin{table}[h!]
	\caption{Comparison of the three algorihms.}
	\label{tab:compare}
        \centering
	\begin{tabular}{|c|c|c|c|}
			\hline
			& Time Complexity & Knowledge of $\sigma(x)$? & $\|f-p_n\|_\infty$ \\
			\hline
			NoisyChebtrunc & $O(N\log N)$ & Not Required & $O(\|\bm{\sigma}\|_\infty\sqrt{\frac{n}{N}})$  \\
			
			Algorithm \ref{alg:noisychebtrunchetero} & $O(N+\hat{N}\log \hat{N})$ & Required & $O(\frac{\|\bm{\sigma}\|_2}{\sqrt{\hat{N}}}\sqrt{\frac{n}{N}})$ \\
			
			HeteroChebtrunc & $O(N+\hat{N}\log \hat{N})$ & Not Required & $O(\frac{\|\bm{\sigma}\|_2}{\sqrt{(1-r)\hat{N}}}\sqrt{\frac{n}{N}})$ \\
			\hline
	\end{tabular}
\end{table}

When developing this algorithm, we also derive a non-asymptotic bound on the sample variance estimator $S^2$. We use Bernstein-type inequalities, which is a class of concentration inequalities commonly used in high-dimensional probability, and bound the absolute error $|S^2-\sigma^2|$ with high probability when the sample random variables are subgaussian. Using this result, we prove that this algorithm improves the error of NoisyChebtrunc by a factor of $ \frac{c}{\sqrt{1-r}}\frac{\|\bm{\sigma}\|_2}{\|\bm{\sigma}\|_\infty \sqrt{\hat{N}}} $, $c>1$ is a constant that can be made arbitrarily close to $ 1$ given $N$ large enough. This can be a significant improvement as the following experiments demonstrate: we approximate the Runge function $f(x) = \frac{1}{1+25x^2}$ under the burst noise $\sigma(x) = \begin{cases}
	1 & \text{if } x \in [0,0.1]\\
	0.00001 & \text{otherwise} 
\end{cases}$, using both the original NoisyChebtrunc and our improved algorithm HeteroChebtrunc\footnote{Unless stated otherwise, we will always use red, blue, green in our figures for NoisyChebtrunc, Algorithm \ref{alg:noisychebtrunchetero}, and HeteroChebtrunc, respectively.}.

\begin{figure}[H]
	\centering
	\includegraphics[width=1\linewidth]{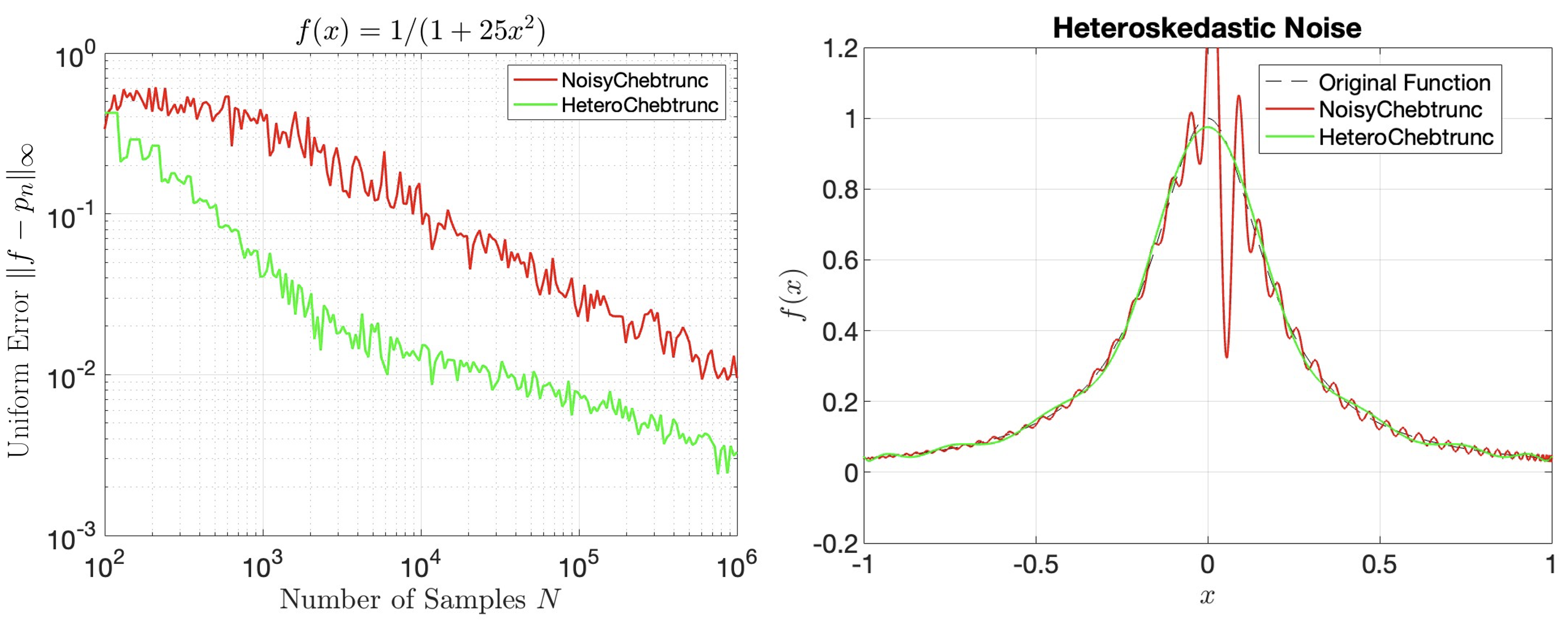}
	\caption{Improved Convergence of HeteroChebtrunc. Left: Uniform Error $\|f-p_n\|_\infty$ v.s. Sample budget $N$. We take 200 logarithmically spaced points from $N = 10^2$ to $N=10^6$, and each data point is averaged from 50 trials. HeteroChebtrunc significantly improves uniform error comparing with NoisyChebtrunc. Right: Approximant from NoisyChebtrunc (red) and HeteroChebtrunc (green) for $N = 10^3$. }
	\label{fig:start}
\end{figure}

We see a significant improvement in uniform error from our algorithm as $N$ becomes moderately large. For $N>10^3$, HeteroChebtrunc has uniform error close to an order of magnitude smaller than NoisyChebtrunc. On the right, we plot an example approximant from NoisyChebtrunc and HeteroChebtrunc for $N = 10^3$. Even with this moderately small sample budget, HeteroChebtrunc consistently produces a good approximant, and the approximant is not noticeably worse in the region $[0,0.1]$ where noise is large, whereas NoisyChebtrunc may sometimes produce a problematic estimate, particularly in the noisy region $[0,0.1]$. This is another key advantage of HeteroChebtrunc compared to other methods as we later discuss in detail in Figure \ref{fig:f-pn}.

We also include other numerical experiments in this paper to demonstrate the reduction in uniform error from our algorithm, and that our theoretical results align with the experiments reasonably well. We highlight here Figure \ref{fig:500trials}, which illustrates that HeteroChebtrunc has uniform error $O(\frac{\|\bm{\sigma}\|_2}{\sqrt{1-r}}\sqrt{\frac{n}{N}})$, as our theoretical analysis in Section \ref{sec:noisy-unknown} predicts.

Finally, we give a formal definition of heteroskedasticity used in this paper:

\begin{definition}\label{def:hetero}
	Given the sampling problem $y=f(x) + \epsilon_x$, where $\{\epsilon_x:x\in [-1,1]\}$ is a family of zero-mean noise random variable such that $\text{Var}(\epsilon) = \sigma(x)^2$ for some function $\sigma(x)\geq 0$. We say that the noise is \emph{homoskedastic} if $\sigma(x) = \sigma_0$ for some constant $\sigma_0$, and \emph{heteroskedastic} otherwise. In this paper, unless otherwise specified, we also assume $\epsilon_x$ and $\epsilon_y$ are independent for $x\neq y$, and different samples taken at the sample $x$ are independent.
\end{definition}

\textit{Notation.} In this paper, the target function is $f$ and the total number of samples budget allowed is $N+1$. The observations are $\{(x_i,y_i)\}$, where $x_i$ are Chebyshev points and

\[y_i = f(x_i) + \epsilon_{x_i},\]
where the noise $\epsilon_{x_i}$ are independent with mean 0 and variance $\sigma(x_i)^2$. $\sigma(x) \geq 0$ is some function in $x\in [-1,1]$. In discussions, we refer to $\sigma_i = \sigma(x_i)$ as the noise level at $x_i$. $\hat{p}_N$ is the interpolant, and $p_n$ is the degree $n$ truncation of $\hat{p}_N$. Vectors are denoted by boldface letters, e.g. $\bm{x} = (x_0, x_1,...,x_N)^T$, and $\| \cdot\|_p$ is the $l_p$ norm, $1\leq p\leq \infty$. $\mathbb{P}$, $\mathbb{E}$ denote probability and expectation, respectively.

$S^2 = \frac{1}{m-1} \sum_{i=1}^{m} (X_i - \bar{X})^2$ denotes the unbiased sample variance estimator, and $\bar{X} = \frac{1}{m}\sum_{i=1}^{m}X_i$ is the mean of i.i.d. samples $X_i$ of a random variable $X$. 

We say $f \ll g$ if $f(x) = O(g(x))$ and vice versa. $f(x) \asymp g(x)$ if $f\ll g$ and $g \ll f$. 

\section{Current Methods for Heteroskedastic noise}\label{sec:lit-review}

Heteroskedastic noise arises in datasets used in a variety of applications, such as in biological sequencing \cite{bio_seq}, radar signal processing \cite{radar}, option pricing \cite{finance}, pharmacokinetic models \cite{pharmacokinetic}, econometrics \cite{econometrics}, and many more. Motivated by this, modern methods has been developed to target heteroskedastic noise in many different scenarios. We discuss a few of such scenarios below.

\begin{enumerate}
	\item Principal Component Analysis (PCA): in PCA, when the sample noises $\epsilon_j$ are heteroskedastic i.e. not uniform, the estimates are known to be inconsistent. In order to address this, HeteroPCA is a state-of-the-art algorithm developed by Cai, Wu and Zhang \cite{PCA}, which uses an iterative approach to address the case of heteroskedastic noise. To summarise its overall idea, HeteroPCA iteratively updates the diagonal elements of its approximant by using a low rank approximation. We highlight an assumption often made in PCA, which is that the sample noises are heteroskedastic and independent \cite{PCA, deflated-PCA, missing-PCA}. HeteroPCA also does not assume that the noise variance is known, which is common in applications. Hence, when studying heteroskedastic noise in our context of function approximation, we develop HeteroChebtrunc that assumes independent heteroskedastic noise, and can adapt to unknown noise level via pre-sampling.
	\item In bootstrapping, when heteroskedastic noise is present, the statistics calculated based on the naïve bootstrap scheme are known to be inconsistent. A modified method, known as the weighted bootstrap, proposed by Wu is given in \cite{weighted_boots}, which instead of drawing resamples uniformly, it draws from a distribution based on heteroskedasticity. This can be shown to produce a heteroskedastic consistent estimator. 
	\item Nonparametric regression is the study of approximating an unknown function from noisy samples in statistics, which is closely related to our work. Under heteroskedasticity, the estimators from many traditional methods have non-uniform variances across different values of $x$. For example, the local polynomial regression estimator $\hat{f}(x)$ will have higher variance on $x$ values with large $\sigma(x)$. In statistics, this leads to larger confidence intervals for $\hat{f}$ \cite{wasserman}. In terms of function approximation, regions with larger variance in estimators tend to have bigger pointwise error, causing the uniform error to scale with the largest noise. This also makes it difficult to obtain an accurate approximation on regions with higher noise. 
	
	Adapting the statistical methods under heteroskedasticity is often difficult, particularly if the form of heteroskedasticity is unknown. Here we briefly outline some of the important methods:
	\begin{enumerate}
		\item In local regression, such as the Nadaraya-Waston kernel estimator \cite{wasserman}, one selects a bandwidth $h$ to control the size of local neighbourhood used to estimator $f(x)$ at each $x$, which is traditionally fixed across $x$. By using a variable bandwidth that adapts to noise variance, one can reduce the effect of heteroskedasticity in the data \cite{fan-and-gilbels,wasserman}.
		\item For dependent noise, different methods have been developed targeting the form of dependence. For example if the noise is spatially correlated, i.e. the covariance between $\epsilon_x$ and $\epsilon_y$ is a function $C(x,y)$ that depends on the $x-$coordinate $x$ and $y$, then an Analysis of Variance (ANOVA) based smoothing spline model can be applied to choose appropriate smoothing parameters that target the correlation between noise \cite{correlated-error}.
		\item When $\sigma(x)$ is not constant, a different approach is to use a least-squares approximant. One modern example is \cite{least_square}, where a hybrid approach that combines Christoffel sampling and least squares is shown to improve accuracy compared with least-squares approximant that ignores heteroskedasticity.
	\end{enumerate}
\end{enumerate}

In the next section, we will explore how viewing from the perspective of numerical analysis can shed new light on this topic, via the well-established theory of Chebyshev interpolation.

\section{Chebyshev Interpolation and NoisyChebtrunc}\label{sec:chebandnoisy}

Chebyshev interpolation is one of the most celebrated algorithms in approximation theory, which samples the target function $f $ at the $N+1$ Chebyshev points $\{x_i\}_{i=0}^N$ where
\begin{align*}
	x_i = \cos(\frac{i\pi}{N}) && i = 0, 1,...,N.
\end{align*}

It then produces an interpolant $\tilde{p}_{N}$ such that $\tilde{p}_{N}(x_i) = f(x_i) $ at each $x_i$. Computing a Chebyshev interpolant relies on the FFT, which enjoys stability as a unitary operation, and speed of just $O(N\log N)$ operations \cite{FFT-Cheb}. Chebyshev points also attain a Lebesgue constant that is within $O(1)$ of best possible, resulting in spectral convergence (exponential for analytic functions, and algebraic for differentiable functions) \cite{ATAP}.  It has guaranteed convergence for any $f$ absolutely continuous (or in general, any $f$ continuous and of bounded variation) as $N$ increases \cite{abs-cont}.

Recall our original problem of approximating a function $f \in C[-1,1]$ using a polynomial, where each evaluation of $f$ comes with noise with distribution $\epsilon_x$, potentially depending on $x$.  We are now in a position to introduce NoisyChebtrunc, which approximates $f$  on $[-1,1]$ in the following manner \cite{yuji}:

\begin{algorithm}
	\caption{NoisyChebtrunc: for approximation of $f$ under i.i.d. noise.}\label{alg:noisychebtrunc}
	\textbf{Input:} an oracle for sampling the noisy univariate function $f$; and $N+1$: the computational budget on the number of samples allowed.
	
	\textbf{Output:} A polynomial $p_n$ of degree $n( < N)$. 
	\begin{algorithmic}[1]
		\item Sample at $N+1$ Chebyshev points $ \{x_i\}_{i=0}^N $ to obtain $\{y_i\}_{i=0}^N$, the noisy data samples.
		\item Find degree $N$ interpolant $\hat{p}_N(x) = \sum_{i=0}^{N}c_iT_i(x)$ such that $\hat{p}_N(x_i) = y_i$, where $T_i$ are Chebyshev polynomials.
		\item Apply Mallow's $C_p$, which is an algorithm that selects a degree $n$ for truncation. 
		\item Truncate $\hat{p}_N$ at $n$, i.e. output $p_n(x) = \sum_{i=0}^{n}c_iT_i(x)$.
	\end{algorithmic}
\end{algorithm}

Here we briefly introduce Mallow's $C_p$. It is widely known that increasing the degree of the interpolant does not necessarily lead to better approximation in the presence of noise \cite{davenport}. Polynomial interpolants are known to suffer from overfitting, and therefore, a degree selection is required. Mallow's $C_p$ is one statistical criterion for this purpose. In statistics, Mallow's $C_p$ is an estimate of the prediction error in least-squares regression models, by essentially penalising the use of too many predictors to address overfitting \cite[Chap. 6]{stats}. NoisyChebtrunc selects a proper degree by selecting $n$ with minimal $C_p$. While the interaction between $C_p$ and Chebyshev interpolation in this context is yet to be fully explored, it has been shown that Mallow's $C_p$ can be relied on to produce a reasonable degree \cite{yuji}, so we will focus more on the approximation theory aspect of this problem rather than $C_p$.

To start our discussion of NoisyChebtrunc, we first need a characterisation of the type of noise. To this end, we define subgaussian and subexponential random variables:

\begin{definition} \cite{Wainwright_2019}[Chap. 2]
	A random variable $X$ is subgaussian with parameter $\sigma$ if 
	\[\mathbb{P}( |X|\geq t) \leq 2\exp(-\frac{t^2}{\sigma^2}).\]
	It can be shown that $\sigma^2$ is the variance of $X$.

	$X$ is subexponential with parameter $\nu, \alpha$ if
	\begin{equation*}
		\mathbb{P}( |X|\geq t) \leq 
		\begin{cases}
			2\exp(-\frac{t^2}{2\nu^2}), &\text{ for } 0\leq t \leq \frac{\nu^2}{\alpha}\\
			2\exp(-\frac{t}{2\alpha}) &\text{ for } t>\frac{\nu^2}{\alpha}
			
		\end{cases}.
	\end{equation*} 
\end{definition}

Let $r_n = f-p_n^*$, where $p_n^*$ is the best degree $n$ polynomial approximation of $f$. Under i.i.d. subgaussian noise of parameter $\sigma$, NoisyChebtrunc produces a polynomial approximant $p_n$ of order $ n $, and the uniform error $\|p_n-f\|_\infty$ is bounded above by $O(\sigma\sqrt{\frac{n}{N}} + \sqrt{n}\|r_n\|_\infty )$ with high probability. Under i.i.d. subexponential noise of parameter $(\nu, \alpha)$, $\|p_n-f\|_\infty$ is bounded above by $O\left( (\frac{\nu^2}{\alpha}\sqrt{\frac{n}{N}}+\sqrt{n}\|r_n\|_\infty)\log n\right)$ with high probability. A precise bound can be found in \cite{yuji}.

\begin{remark}
	Since $\|r_n\|_\infty$ is known to decay at a spectral rate as it is the error of the best polynomial approximant \cite{ATAP}, this means NoisyChebtrunc will converge at a spectral rate until the error reaches $O(\sigma \sqrt{\frac{n}{N}})$ in the subgaussian case, after which the first term will dominate and the error decays like $O(\frac{1}{\sqrt{N}})$. Overall, one expects the error to be $O(\sigma\sqrt{\frac{n}{N}})$.
\end{remark}

Now consider the heteroskedastic case as in Definition \ref{def:hetero}. For NoisyChebtrunc, if $N+1$ Chebyshev sample points $\{x_i\}_{i=0}^N$ is used, let $ \bm{\sigma} = (\sigma_0, ..., \sigma_N)$ be defined as $\sigma_i = \sigma(x_i)$, i.e. the error level at each Chebyshev point. The error will be proportional to $\frac{\|\bm{\sigma} \| _\infty}{\sqrt{N}}$ \cite{yuji}. This means in order to improve NoisyChebtrunc under heteroskedasticity, we need to reduce the maximum noise level at the Chebyshev points. In the next section, we show how a simple weighted sampling approach can achieve this reduction and lead to better uniform error.

\section{Heteroskedasticity and Weighted Sampling}\label{sec:noisy-hetero}

In this section, we propose a variant of NoisyChebtrunc called Algorithm \ref{alg:noisychebtrunchetero},  which aims to address the case where noise is heteroskedastic, i.e. when the variance $ \sigma(x)^2 $ of noise is not uniform. Crucially, this algorithm requires $ \sigma(x) $ as an input, a condition we will later relax in Algorithm \ref{alg:noisychebtrununknown}. We derive a uniform error bound for Algorithm \ref{alg:noisychebtrunchetero}, and show that it has a uniform error at least as good at NoisyChebtrunc and improves upon it when noise is heteroskedastic for $N$ sufficiently large.

\subsection{Redistribution of Noise}
To motivate this algorithm, we assume the form of heteroskedasticitiy as defined in Definition \ref{def:hetero}. Thus, it is possible to sample at this point $k_i$ times for some $k_i$, and average the samples taken to yield $y_i$. This reduces the noise level at $x_i$ from $\sigma(x_i)$ to $\frac{\sigma(x_i)}{\sqrt{k_i}}$. Naturally, one can consider the following simple variant of NoisyChebtrunc mentioned in \cite{yuji}:

Consider $N+1$ as the total number of samples allowed and $\hat{N}+1$ the number of Chebyshev sample points, where $\hat{N} \ll N$ and for simplicity we assume $k = (N+1)/(\hat{N}+1)$ to be integer:

\begin{enumerate}
	\item Take $k$ samples at each of the $\hat{N}+1$ Chebyshev points, and let $y_i$ be the average of the $k$ samples at $x_i$, $0\leq i\leq \hat{N}$.
	\item Perform Chebyshev interpolation using data $(x_i,y_i)$, and truncate degree using Mallow's $C_p$.
\end{enumerate}

It is not hard to see that it has similar performance to NoisyChebtrunc: if the noise level $\sigma(x) = \sigma_0$ is homoskedastic, then the noise at each data point is $\frac{\sigma_0}{\sqrt{k}}$, so this variant will have spectral rate of noise reduction until error reaches $O(\frac{\sigma_0}{\sqrt{k}}\sqrt{\frac{n}{\hat{N}}}) = O(\sigma_0\sqrt{\frac{n}{N}})$, which if we assume $n$ is fixed is same as the noise reduction of the original NoisyChebtrunc. 

Recall in the discussion ending the previous section, where in order to improve the uniform error, one needs to reduce the maximum error experienced at each Chebyshev point. Using the multiple sample approach, one way to achieve this objective is to sample the same point multiple times like the aforementioned variant, but instead of sampling the same number of times at each Chebyshev point, we sample based on the level of noise at each point: we assign more samples at noisier nodes and fewer samples at cleaner nodes. This way, we \emph{redistribute} the noise so that in the end the averaged samples have reduced maximum noise level. To make this precise, we propose the following:

\begin{algorithm}
	\caption{Special case of HeteroChebtrunc under known $\sigma(x)$.}
	\label{alg:noisychebtrunchetero} 
	\textbf{Input:} An oracle for sampling the noisy univariate function $f$; $N+1$: the computational budget on number of samples allowed; $\sigma(x)$: the noise level.
	
	\textbf{Output:} A polynomial $p_n$ of degree $n( < N)$.
	\begin{algorithmic}[1]
		\item Choose $\hat{N} \ll N$. Compute $\bm{\sigma} = (\sigma_0, \sigma_1,...,\sigma_{\hat{N}})$ for the $\hat{N}+1$ Chebyshev points $\{ x_i \}_{i=0}^{\hat{N}}$.
		\item For each $x_i$, $ 0\leq i \leq \hat{N}$,  sample $f(x_i)$ $k_i$ times, where $k_i\geq 1$ and $k_i$ is proportional to $\sigma_i^2$. Define $y_i$ as the average of the samples. 
		\item Interpolate at the points $(x_i, y_i)$ to find the degree $\hat{N}$ interpolant $\hat{p}_{\hat{N}}$. Truncate using Mallow's $C_p$ as in NoisyChebtrunc. 
	\end{algorithmic}
\end{algorithm}

Here we explain Step 2 in detail. At $x_i$, the averaged sample has variance $ \frac{\sigma^2_i}{{k_i}} $, and let $\bm{\bar{\sigma}} = (\frac{\sigma^2_1}{k_1}, ..., \frac{\sigma^2_{\hat{N}}}{{k_{\hat{N}}}} )$. In order to minimise $\max \frac{\sigma^2_i}{{k_i}}$, $k_i$ needs to be chosen such that $\frac{\sigma^2_i}{k_i}$ are roughly equal. This requires $k_i$ to be chosen proportionally to $\sigma_i^2$, that is, $k_i = (N+1)\frac{\sigma_i^2}{\sum \sigma_j^2}$ (ignoring rounding).

We briefly discuss the time complexity of this algorithm. The weighting calculations can be completed in $O(N) $, and by using FFT, Chebyshev interpolation on $\hat{N}+1$ points requires only $\hat{N}\log \hat{N}$ operations, so the overall complexity is simply $O(N+\hat{N}\log\hat{N})$, which is even slightly faster than NoisyChebtrunc's $O(N\log N)$.

We present a uniform error bound of Algorithm \ref{alg:noisychebtrunchetero} below:

\begin{theorem}\label{thm:error_bound_hetero}
	Let $f$ be a noisy function on $[-1,1]$ where the noise the heteroskedastic, as in Definition \ref{def:hetero}. Let $N+1$ be the total number of sample budgets, and $p_n$ be the polynomial approximant produced by Algorithm \ref{alg:noisychebtrunchetero} with input $f$ and $N+1$, and sampling is carried out at the $\hat{N}+1$ Chebyshev points $\{x_i\}_{i=0}^{\hat{N}}$. If $\epsilon_{x_i}$ is subgaussian with parameter $ \sigma(x_i)=: \sigma_i $, then for large $N$, for any fixed $x\in [-1,1]$,
	\[\mathbb{P} \left[|p_n(x)-f(x)| > 2t\frac{\|\bm{\sigma}\|_2}{\sqrt{N\hat{N}}} \sqrt{n+1} + (\sqrt{8(n+1)}+1)\|r_n \|_\infty\right] \leq 2\exp(-\frac{t^2}{2}).\] 
	
\end{theorem}

\begin{proof}
	For $N$ large enough, we can assume $k_i = N \frac{\sigma_i^2}{\sum \sigma_j^2} = N\frac{\sigma_i^2}{\|\bm{\sigma}\|_2^2}$, so using the properties of subgaussian random variables:
	\[ \bar{\sigma}_i = \frac{\sigma_i}{\sqrt{k_i}} = \frac{\sigma_i \|\bm{\sigma}\|_2}{\sqrt{N}\sigma_i} = \frac{\|\bm{\sigma}\|_2}{\sqrt{N}}. \]
	This means that the redistributed noise now has uniform variance, and since the sum of independent subgaussian random variables is still subgaussian, the result follows directly from Theorem 4.1 in \cite{yuji}.
\end{proof}

Below is the analogous result for subexponential distributions:

\begin{theorem}\label{thm:error_bound_hetero_exp}
	Let $f$ be a noisy function on $[-1,1]$ where the noise the heteroskedastic, as in Definition \ref{def:hetero}. Let $N+1$ be the total number of sample budgets, and $p_n$ be the polynomial approximant produced by Algorithm \ref{alg:noisychebtrunchetero} with input $f$ and $N+1$, and sampling is performed at the $\hat{N}+1$ Chebyshev points $\{x_i\}_{i=0}^{\hat{N}}$. If $\epsilon_{x_i}$ is subexponential with parameter $(\nu_i, \alpha_i)$, then for large $N$, for any fixed $x\in [-1,1]$,
	
	\[\mathbb{P} \left[|p_n(x)-f(x)|> \left(\frac{2}{\pi} \log(n+1)+1 \right) \sqrt{n+1} \left(2t\frac{\|\bm{\nu}\|_2^2}{\alpha} \frac{1}{\sqrt{N}}+\sqrt{8}\|r_n\|_\infty \right)\right]\]
	\[\leq 2(n+1)\exp(-\frac{\|\bm{\nu}\|_2^2}{2\alpha^2}t^2)\]
	for $0 \leq t\leq t_*$, and 
	\[\mathbb{P} \left[|p_n(x)-f(x)|> \left(\frac{2}{\pi} \log(n+1)+1 \right) \sqrt{n+1} \left(2t\frac{\|\bm{\nu}\|_2^2}{\alpha} \frac{1}{\sqrt{N}}+\sqrt{8}\|r_n\|_\infty \right)\right]\]
	\[\leq 2(n+1)\exp(-\frac{\|\bm{\nu}\|_2^2}{2\alpha^2}t)\]
	for $t \geq t_*$, where $ \bm{\nu} = (\nu_0,\dots, \nu_{\hat{N}})$, $ \alpha = \max_i \alpha_i$.
\end{theorem}

\begin{proof}
	The proof uses the exact same reasoning as Theorem \ref{thm:error_bound_hetero}, using the identity that $\sum_i X_i$ is subexponential with parameter $(\sqrt{\sum_i \nu_i^2}, \max_i\alpha_i)$ if $X_i$ are independent subexponential of parameter $(\nu_i,\alpha_i)$.
\end{proof}

The above results demonstrate that under heteroskedastic noise level $\sigma(x)$, the error of HeteroChebtrunc decays like $O(\frac{\|\bm{\sigma}\|_2}{\sqrt{N\hat{N}}})$, which is smaller than NoisyChebtrunc's $O(\frac{\|\bm{\sigma}\|_\infty}{\sqrt{N}})$:

\[\|\bm{\sigma}\|_2 = \sqrt{\sum_{i=0}^{\bar{N}} \sigma_i^2} \leq \sqrt{\sum_{i=0}^{\bar{N}} \|\bm{\sigma}\|_\infty^2} =  \sqrt{(\hat{N}+1)}\|\bm{\sigma}\|_\infty,\]
where equality is achieved if and only if $\sigma$ is constant.

Therefore, $\frac{\|\bm{\sigma}\|_2}{\sqrt{N(\hat{N}+1)}} \leq \frac{\|\bm{\sigma}\|_\infty}{\sqrt{N}}$, and the inequality is strict if $\sigma_i$ are not uniform, i.e. noise is heteroskedastic. This suggests that Algorithm \ref{alg:noisychebtrunchetero} is likely to outperform NoisyChebtrunc in these cases. 

In addition, for large $N$ such that the uniform error has reached the noise level, we expect our adaptation to improve the error by a factor of $\frac{\|\bm{\sigma}\|_2}{\|\bm{\sigma}\|_\infty\sqrt{\hat{N}}}$. We now demonstrate this fact through numerical experiments. Our implementations use the Chebfun toolbox \cite{chebfun} and can be found in the GitHub repository https://github.com/InigoMontoya314/HeteroChebtrunc.

\subsection{Numerical Experiments}\label{sec:num-hetero}

\subsubsection{Redistributed Noise of Algorithm \ref{alg:noisychebtrunchetero}}

We first demonstrate the noise redistribution effect of Algorithm \ref{alg:noisychebtrunchetero} as illustrated in Figure \ref{fig:redistnoise}. The experiment is conducted with the Runge function $f(x) = \frac{1}{1+25x^2}$ and error function $\sigma(x) = \begin{cases}
	1 & \text{if } x \in [0,1]\\
	0.00001 & \text{otherwise} 
\end{cases}$. Algorithm \ref{alg:noisychebtrunchetero} was applied with\footnote{This choice was not important. As will be discussed in later section, $\hat{N}$ has little impact on the uniform error as long as it is larger than the degree chosen by Mallow's $C_p$.} $\hat{N} = \sqrt{N}$.

\begin{figure}[h]
	\centering
	\centerline{\includegraphics[width=1.0\linewidth]{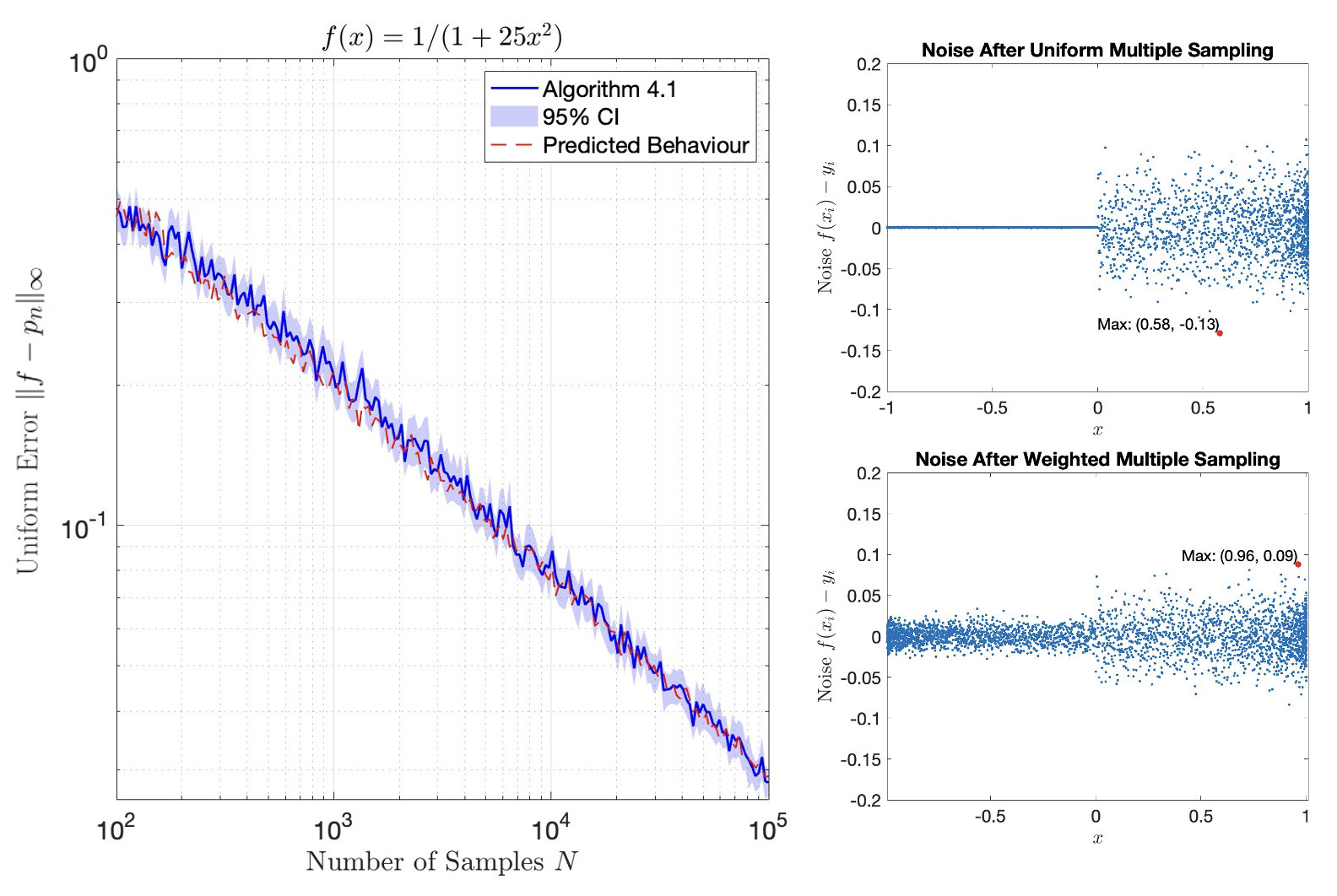}}
	\caption{Noise Redistribution Effect of Algorithm \ref{alg:noisychebtrunchetero}. Left: Uniform Error v.s. Number of Samples $N$ of Algorithm \ref{alg:noisychebtrunchetero} (blue), and of NoisyChebtrunc under the constant noise level $\frac{\|\bm{\sigma}\|_2}{\sqrt{\hat{N}}}$ (red dashed).
	Top right: Noise distribution $f(x_i) - y_i$ after multiple sample, where number of samples at each $x_i$ is equal. Bottom right: number of samples assigned by Algorithm \ref{alg:noisychebtrunchetero}. The weighted sampling redistributes noise and "evens out" the noise level, thus reducing the maximum noise.}
	\label{fig:redistnoise}
\end{figure}

According to the analysis in the previous section, for $N$ large enough Algorithm \ref{alg:noisychebtrunchetero} should demonstrate a convergence behaviour similar to NoisyChebtrunc with ${N+1}$ sample points, where the error is uniform and of parameter $ \frac{\|\bm{\sigma}\|_2}{\sqrt{\hat{N}}} $. This behaviour is confirmed by the experiment represented in the left panel of Figure \ref{fig:redistnoise}. We compute the uniform error of the approximants of both scenarios at 200 logarithmically spaced points between $N = 10^2$ and $N = 10^6$, with each data point being the average of 50 trials. It is clear that for large $N$, the predicted behaviour (red dashed) is well contained in the 95\% confidence interval (blue shaded), so Algorithm \ref{alg:noisychebtrunchetero} has a uniform error similar to NoisyChebtrunc with the constant noise level $\frac{\|\bm{\sigma}\|_\infty}{\sqrt{\hat{N}}}$, as expected in Theorem \ref{thm:error_bound_hetero}.

We also illustrate the noise redistribution effect of Algorithm \ref{alg:noisychebtrunchetero} via the figure on the right panel. Both sampling were conducted with $N = 10^7$, and the same $f$ as above. Here  $\sigma(x) = \begin{cases}
	1 & \text{if } x \in [0,1]\\
	0.01 & \text{otherwise} 
\end{cases}$ for clearer demonstration. The unweighted multiple sampling variant suggested in \cite{yuji} does not take into account the heteroskedasticity, while Algorithm \ref{alg:noisychebtrunchetero} weights the number of samples at each Chebyshev node based on the noise level $\sigma(x_i)$, thereby "evening out" the noise distribution, achieving the objective of reducing maximum noise level. Intuitively, $\|\bm{\sigma}\|_\infty$ is the maximum noise level on $[-1,1]$, while $ \frac{\|\bm{\sigma}\|_2}{\sqrt{\hat{N}}}$ is an average noise level across all samples, so if the samples are very noisy in one region, but overall has low average noise level, then this redistribution process allows one to significantly reduce maximum noise sampled and improve uniform error. 

\subsubsection{Comparison with NoisyChebtrunc}

Next we compare the uniform error of Algorithm \ref{alg:noisychebtrunchetero} against NoisyChebtrunc. 
The experiments at Figure \ref{fig:heterovsoriginal} are conducted with the Runge function $f(x) = \frac{1}{1+25x^2}$ with two different noise functions: the left panel with $\sigma(x) = \begin{cases}
	1 & \text{if } x \in [0,1]\\
	0.00001 & \text{otherwise} 
\end{cases}$, which is a simple demonstration of heteroskedasticity where the nodes are noisier on $[0,1]$ and cleaner on $[-1,0)$. The data on the right panel is computed with noise $\sigma(x) = \begin{cases}
	10 & \text{if } x \in [0.9,1]\\
	0.00001 & \text{otherwise}
\end{cases}$, which is designed to be a burst noise case, with high noise on a small interval and low noise everywhere else\footnote{All numerical experiments conducted simulates noise with the normal distribution, and we expect our conclusions from our experiments to hold for general subgaussian noise.}.

We again compute the uniform error $\| f - p_n\|_\infty$ of the polynomial approximant yielded by NoisyChebtrunc and Algorithm \ref{alg:noisychebtrunchetero}, on 200 logarithmically spaced values from $N = 10$ to $N = 10^5 $, and the choice of $\hat{N} = 3\sqrt{N}$ . On both noise cases Algorithm \ref{alg:noisychebtrunchetero} (blue) is seen to achieve lower uniform error than NoisyChebtrunc (red), with perhaps more significant improvements in the burst noise case.

\begin{figure}[H]
	\centering
	\centerline{\includegraphics[width=1\linewidth]{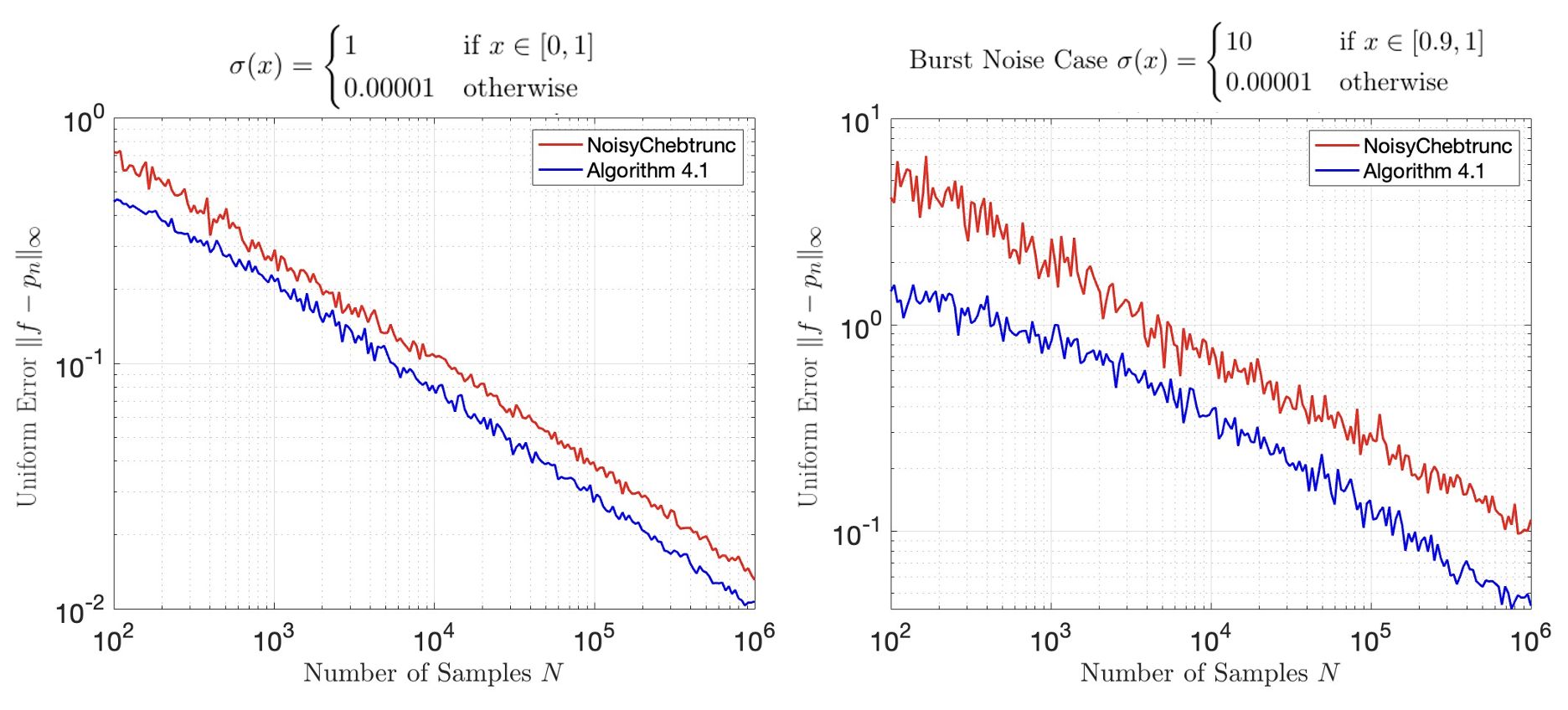}}
	\caption{Comparison of Convergence Behavior with NoisyChebtrunc. Left: $ \sigma(x) = \mathbbm{1}_{[0,1]} + 0.00001\mathbbm{1}_{[-1,0)} $. Right: $ \sigma(x) = 10\mathbbm{1}_{[0.9,1]} + 0.00001\mathbbm{1}_{[-1,0.9]} $. The plot compares uniform error of Algorithm \ref{alg:noisychebtrunchetero} (blue) with NoisyChebtrunc (red) on a log-log plot, and each data point is averaged from 50 independent trials.}
	\label{fig:heterovsoriginal}
\end{figure}

We note another observation that was not included in \cite{yuji}. As seen in Figure \ref{fig:heterovsoriginal} the decay of the uniform error is not monotone, but rather highly oscillatory, with the uniform error oscillating frequently even though the overall trend is decreasing. We run this algorithm with a very large trial size (1000 trials) and still observed similar oscillation, though larger trial size does lead to reduction of its amplitude in log-log scale. This phenomenon merits further investigation, but as here we mainly focus on analysing the proposed extensions of NoisyChebtrunc we will not study this here.

\begin{figure}[H]
	\centering
	\centerline{\includegraphics[width=1.1\linewidth]{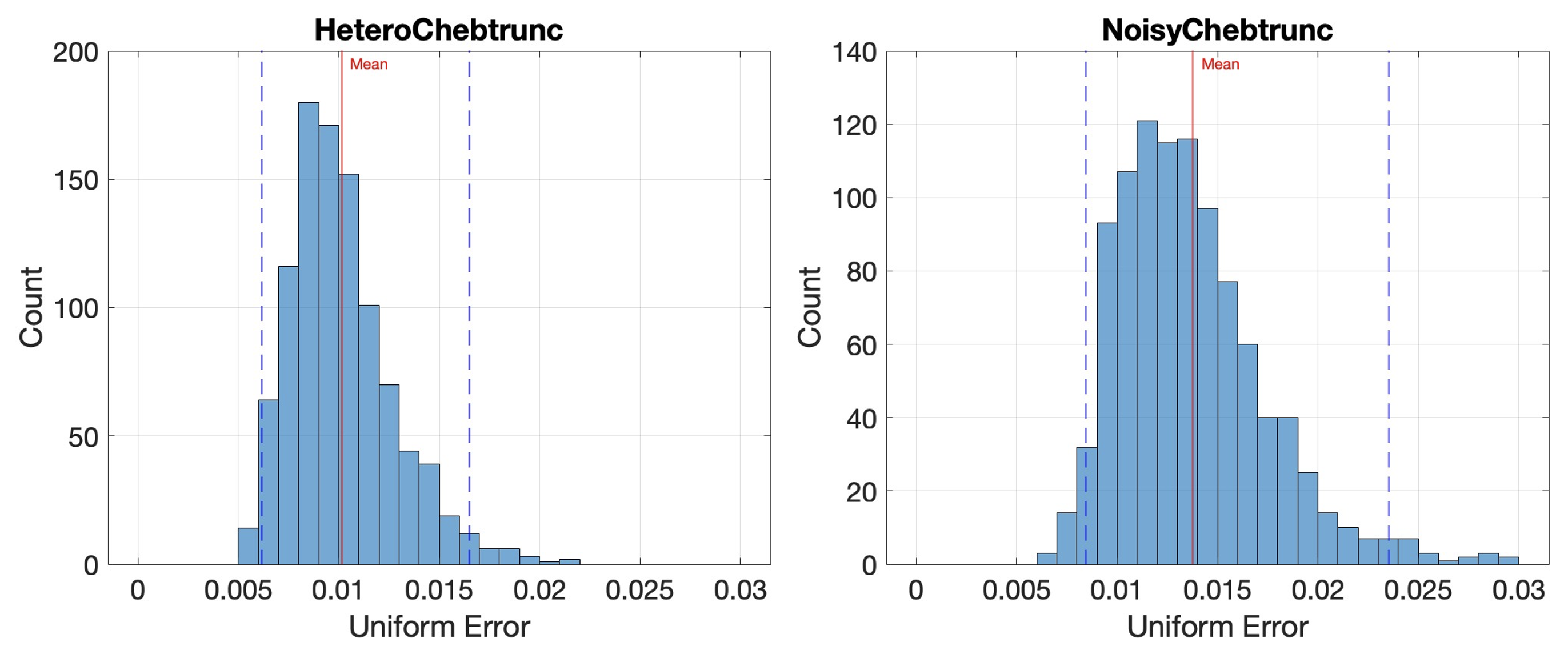}}
	\caption{Histogram of 1000 trials of Algorithm \ref{alg:noisychebtrunchetero} (left) v.s. NoisyChebtrunc (right), with mean (red line) and 2.5th and 97.5th percentile (blue dashed). }
	\label{fig:histogramhetero}
\end{figure}

We also plot the histogram of 1000 independent runs of Algorithm \ref{alg:noisychebtrunchetero} and compare it with NoisyChebtrunc. Both algorithms are applied to the Runge function $f(x) = \frac{1}{1+25x^2}$ with noise function  $\sigma(x) = \begin{cases}
	1 & \text{if } x \in [0,1]\\
	0.00001 & \text{otherwise} 
\end{cases}$, and $N = 10^6$. 

In both instances, the uniform errors are concentrated around its mean, with Algorithm \ref{alg:noisychebtrunchetero} achieving overall smaller error than the original algorithm, which is expected from previous analysis. We observe that for this noise function, $\frac{\|\bm{\sigma}\|_2}{\|\bm{\sigma}\|_\infty \sqrt{\hat{N}}} \approx \frac{1}{\sqrt{2}}$. If we multiply the mean and confidence intervals for the NoisyChebtrunc case by this factor, then the result is close to the ones calculated for Algorithm \ref{alg:noisychebtrunchetero}. This confirms the previous analysis that for large $N$, our extension reduces the uniform error by the aforementioned factor. This also means that Algorithm \ref{alg:noisychebtrunchetero} has uniform error more concentrated around its mean than the original algorithm.

\subsection{Choice of the Number of Chebyshev Points $\hat{N}$}\label{sec:nhat-choice}

There is one issue that remains unaddressed, which is the choice of $\hat{N}$ in Algorithm \ref{alg:noisychebtrunchetero}. 

We recall from Theorem \ref{thm:error_bound_hetero}, the infinity norm error of Algorithm \ref{alg:noisychebtrunchetero} is $O(\|\bm{\sigma}\|_2\sqrt{\frac{n}{\hat{N}N}}+\sqrt{n}\|r_n\|_\infty)$, and when $\sigma$ is constant, this is the same as the original NoisyChebtrunc. Therefore, we do not expect the choice of $\hat{N}$ to influence the uniform error too much.

However, we do have two requirements for the choice: (i) We need $\hat{N}$ to be larger than the degree of truncation $n$ in NoisyChebtrunc, chosen by Mallow's $C_p$; (ii) we need $\hat{N}$ to be sufficiently large so that the uniform error in NoisyChebtrunc reaches the noise level. Our experiment in Figure \ref{fig:nhatchoice} shows as long as both can be achieved, the choice of $\hat{N}$ is not important:

\begin{figure}[H]
	\centering
	\centerline{\includegraphics[width=1\linewidth]{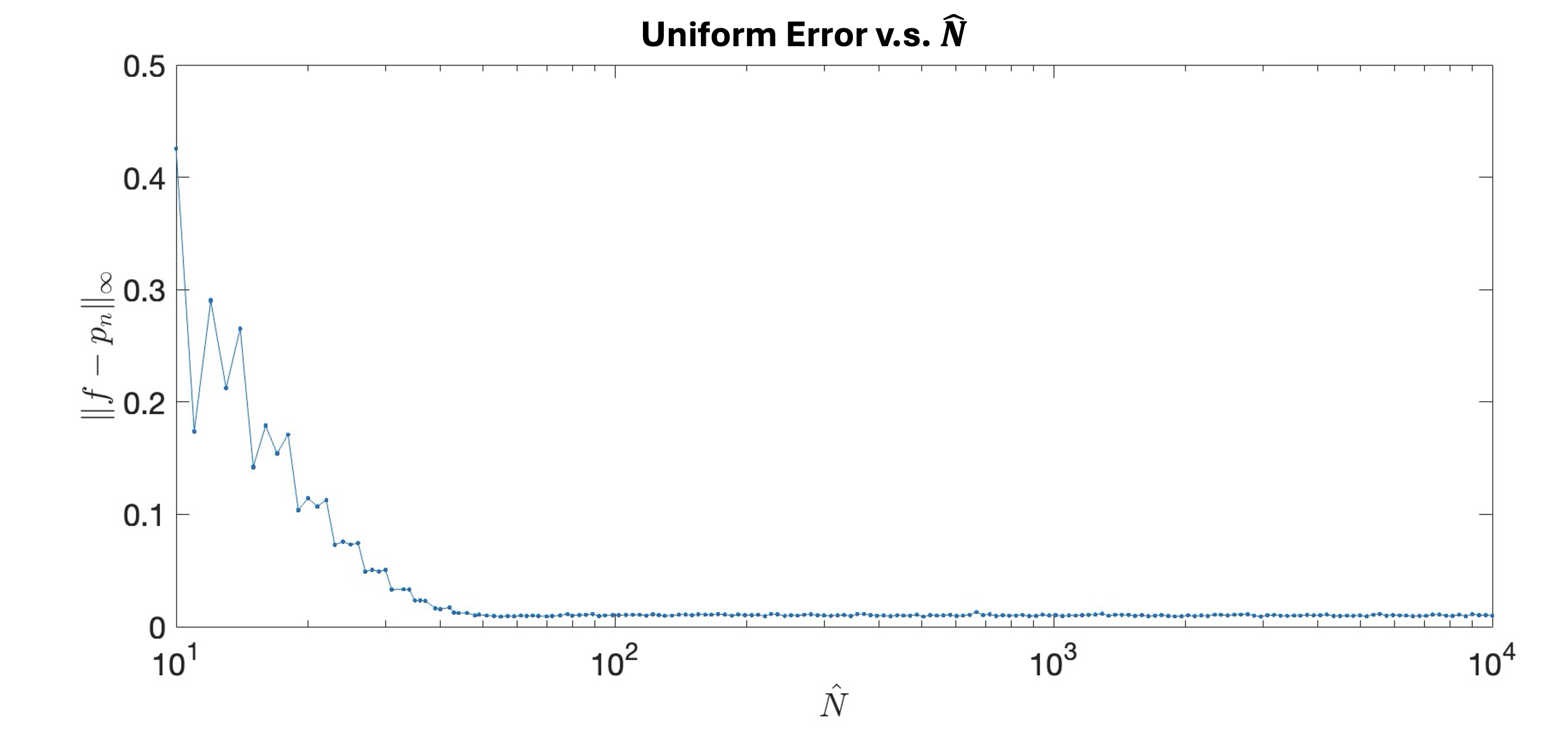}}
	\caption{Influence of the number of Chebyshev points $\hat{N}$ on uniform error. We fix $N=10^6$, and vary $\hat{N}$ from $10$ to $10^4$, in 200 logrithmically spaced points, and compute the uniform error $\|f-p_n\|_\infty$ of the approximant from Algorithm \ref{alg:noisychebtrunchetero}. Each data point is averaged from 50 independent trials. The choice of $\hat{N}$ does not impact the uniform error for $\hat{N}$ except for very small $\hat{N}$. It should also be noted that as $\hat{N}$ approaches $N$ (not shown), the uniform error increases again (see footnote).}
\label{fig:nhatchoice}
\end{figure}

We use Algorithm \ref{alg:noisychebtrunchetero} to compute approximants of the Runge function $f(x) = \frac{1}{1+25x^2}$ with $N = 10^6$ and noise function $\sigma(x) = \begin{cases}
	1 & \text{if } x \in [0,1]\\
	0.00001 & \text{otherwise} 
\end{cases}$. We use $\hat{N}$ ranging from $10$ to $10^4$ (left panel), and plot the uniform error $\|f-p_n\|_\infty$ against $\hat{N}$.\footnote{If $\hat{N}$ is too close to $N$ then we cannot effectively apply weighted sampling, so we are essentially applying the original NoisyChebtrunc, losing the ability to even out the maximum noise. Hence $\hat{N}$ should be chosen much smaller than $N$.} The uniform error are all on the same order of magnitude, except when $\hat{N}$ is less than $ 50 $. This shows the choice of $\hat{N}$ does not significantly impact the convergence for $\hat{N}$ moderately large, and in this paper we have chosen $\hat{N} = O(\sqrt{N})$, and from our experience this is more than enough to guarantee optimal uniform error.

\section{Unknown Noise Level and HeteroChebtrunc}\label{sec:noisy-unknown}

We have shown that Algorithm \ref{alg:noisychebtrunchetero} effectively improves the infinity-norm error of NoisyChebtrunc, but has one major drawback: it requires knowledge of $\sigma(x)$, or at least $ \{\sigma_i\}_{i=0}^{\hat{N}} $ where $\sigma_i = \sigma(x_i)$ are noise levels at Chebyshev points. This is usually not known in practice. However, one simple remedy for this is to \textbf{pre-sample} at each point to compute a variance estimate $S_i^2$ for the noise at each point, then apply Algorithm \ref{alg:noisychebtrunchetero}. The pre-samples are not wasted either, as they can be used as part of the weighted average when computing $y_i$. In this section, we present an algorithm called HeteroChebtrunc when the noise is independent, but heteroskedastic and $\sigma(x)$ is unknown. It includes a pre-sampling procedure to estimate $\frac{\sigma_i^2}{\sum \sigma_j}$ for $0\leq i\leq \hat{N}$. In order to determine the size and accuracy of our pre-samples, we derive a non-asymptotic bound for the sample variance estimator $S^2$ when the sample random variables are i.i.d. subgaussian. This allows us to study the error $ \left|\frac{S_i^2}{\sum_{j=0}^{\hat{N}}S_i^2} - \frac{\sigma_i^2}{\sum_{j=0}^{\hat{N}}\sigma_i^2}\right|$ with high probability. We show that HeteroChebtrunc reduces the maximum noise effect to at most $c\frac{\|\bm{\sigma}\|_2}{\sqrt{N}}$ with high probability, where $c>1$ is a constant that can be made arbitrarily close to 1 given large $N$. Therefore, HeteroChebtrunc will have a similar noise redistribution effect to Algorithm \ref{alg:noisychebtrunchetero}. Hence, this algorithm will have a smaller uniform error $\|f-p_n\|_\infty$ than NoisyChebtrunc, and does not require any knowledge of the noise level $\sigma(x)$ like Algorithm \ref{alg:noisychebtrunchetero}, making it a competitive choice when noise is heteroskedastic and unknown. We demonstrate our findings via numerical experiments.

\subsection{Determining the Noise Level: Pre-sampling}

We first state our algorithm formally:

\begin{algorithm}
	\caption{HeteroChebtrunc for approximation of $f$ under unknown heteroskedastic noise.}\label{alg:noisychebtrununknown}
	\textbf{Input:} Given an oracle for sampling the noisy univariate function $f$ as in Definition \ref{def:hetero}; and $N+1$: the computational budget on the number of samples allowed.
	
	\textbf{Output:} A polynomial $p_n$ of degree $n( < N)$.
	\begin{algorithmic}[1]
		\item Choose parameters $\hat{N}$ and $r$: $\hat{N}+1 \ll N$ is the number of Chebyshev points, and $0< r< 1$ is the proportion of sampling budget allocated for pre-sampling. Define $N_1 = rN$, and pre-sample size as $m = N_1/(\hat{N}+1)$ (assumed an integer). Default $r=0.1$, $\hat{N} = \lfloor\sqrt{N}\rfloor$.
		\item For each $x_i$, $ 0\leq i \leq \hat{N}$, sample $f(x_i)$ $m$ times to get $\{y_{i,j}\}_{j=1}^m$, and compute the sample variance estimator $S_i^2 = \frac{1}{m-1} \sum_{j=1}^m(y_{i,j} - \bar{y_i})^2$ at each $x_i$, where $\bar{y_i} = \frac{1}{m}\sum_{j=1}^{m} y_{i,j}$. 
		\item For each $x_i$, $ 0\leq i \leq \hat{N}$, define $\hat{k_i} = \max\{0, \frac{S_i^2}{\sum_{j=0}^{\hat{N}}S^2_j}(N+1-N_1) -m\}$, the number of samples at $x_i$ needed after pre-sampling, distributed proportional to $S^2_i$. \footnote{Note we don't require $\hat{k}_i \geq 1$ because we have already taken pre-samples, so no risk of not sampling any $x_i$.}
		\item Take $k_i$ samples at each $x_i$, compute $y_i$ as the mean of all taken samples at $x_i$, including pre-samples. 
		\item Interpolate at the points $(x_i, y_i)$ to find the degree $\hat{N}$ interpolant $\hat{p}_{\hat{N}}$. Truncate using Mallow's $C_p$ as in NoisyChebtrunc. 
	\end{algorithmic}
\end{algorithm}

As explained in the introduction of this section, we first allocate $N_1$ samples from our total budget, and distribute them uniformly among the Chebyshev points $ \{x_i\}_{i=0}^{\hat{N}}$. We then use the unbiased sample variance estimator
\[S_i^2 = \frac{1}{m-1} \sum_{j=1}^{m} (X_j - \bar{X})^2,\]
where $\{X_j\}_{j=1}$ are samples taken at $x_i$, and $\bar{X} = \frac{1}{m}\sum_{j=i}^{m} X_j$ is the sample mean. For our case, we assume the samples $X_j$ taken at a point $x$ have noises that are independent and subgaussian. To be precise, we assume $\{X_j\}_{j=1}^m$ are i.i.d. subgaussian with mean $f(x) $ and variance $\sigma(x)$.

Using the estimations $S_i^2$, we then distribute the rest of sampling budget $N+1 - N_1$ to the Chebyshev points, proportional to $\frac{S_i^2}{\sum S_j^2}$ as in Algorithm \ref{alg:noisychebtrunchetero}, define $y_i$ as the mean of all samples (pre-samples included so that they are not wasted)  taken at $x_i$. then interpolate and truncate using Mallow's $C_p$ as is standard in our approach. In total, at each $x_i$, $k_i = \hat{k}_i+ m$ samples are taken.

We draw the attention of the reader briefly to the computational complexity of this algorithm: all the pre-sampling, estimation and weighted sampling can be performed in $O(N)$ operations, so it essentially has the same $O(
\hat{N}\log \hat{N}+N)$ time complexity as Algorithm \ref{alg:noisychebtrunchetero}.

\subsection{Accuracy of the Sample Variance Estimator} \label{sec:sample_var}

Recall in order to apply Algorithm \ref{alg:noisychebtrunchetero}, we take a weighted sample approach by taking approximately $N \frac{\sigma^2_i}{\sum_j \sigma^2_j}$ samples at $x_i$. When noise level is unknown, at each Chebyshev point $x_i$, say $x$, we take $m$ pre-samples $X_1,..., X_m$ and compute their sample variance.

Our problem is to find a bound on the absolute error $|S^2 - \sigma(x)^2|$ for a given $m$ and $x$ locally, which can then be used to determine the error of our estimate $|\frac{S^2_i}{\sum S^2_j} - \frac{\sigma_i}{\sum \sigma_j}|$. While the result feels basic, the authors are unable to find such a bound in the current literature, thus we prove one here as Lemma \ref{lem:sample_var_bound}. In the process, we apply a class of concentration inequalities called Bernstein-type bounds, which are commonly used in high-dimensional probability.

\begin{lemma}\label{lem:sample_var_bound}
	Suppose $ X$ is subgaussian of parameter $\sigma$, $X_1, ...,X_m$ i.i.d. realisations of X, and let $S^2 = \frac{1}{m-1} \sum_{i=1}^{m} (X_i-\bar{X})^2$, the sample variance estimator using $m$ samples. Then $S^2$ is subexponential of parameter $(\nu, \alpha)$, where $\nu \leq \frac{4\sigma}{\sqrt{m}}(1+1/\sqrt{m-1})$, $\alpha = \frac{4\sigma}{m}(1+1/\sqrt{m-1})$. Hence we have the bound for $t>0$, 
	
	\[\mathbb{P}(|S^2-\sigma^2| > t) \leq \begin{cases}
		2\exp(-\frac{t^2}{2\nu^2}) & \text{if } 0\leq t\leq \frac{\nu^2}{\alpha}\\
		2\exp(-\frac{t}{2\alpha}) & \text{if } t>\frac{\nu^2}{\alpha}
	\end{cases}.\]
	Or equivalently, 
	
	\[\mathbb{P}(|S^2-\sigma^2| > t) \leq \max \{		2\exp(-\frac{t^2}{2\nu^2}),
	2\exp(-\frac{t}{2\alpha})\}	.\]
\end{lemma}

%
%
%

\begin{proof}
	From the standard properties of subgaussian and subexponential random variables, for each $0\leq i\leq m$,
	\[X_i-\bar{X} = \frac{m-1}{m}X_i - \frac{1}{m}\sum_{j\neq i} X_j\] is subgaussian of parameter $\frac{m-1}{m}\sigma +\frac{\sqrt{m-1}}{m}\sigma$.
	
	If $X$ is subgaussian of parameter $\sigma$, then $X^2-\mathbb{E}[X^2]$ is subexponential of parameter $(4\sigma,4\sigma)$ \cite{MIT_high_dim_stats}, from this one obtains $S^2 = \frac{1}{m-1}\sum_{i=1}^{m} (X_i-\bar{X})^2$ is subexponential of parameter $(\nu, \alpha)$, where 
	\[\begin{cases}
		\nu \leq \frac{4\sigma}{\sqrt{m}}(1+\frac{1}{\sqrt{m-1}})\\
		\alpha = \frac{4\sigma}{m}(1+\frac{1}{\sqrt{m-1}})
	\end{cases}.\]
	
	The result now follows directly from definition of subexponential random variable.
\end{proof}

Using the above lemma, we derive a bound on the relative error of our estimation $\frac{S_i^2}{\sum S_j^2}$, which we later use to study the noise sampled in HeteroChebtrunc. We state it here as a proposition:

\begin{proposition}\label{cor:var_proportional_bound}
	For each $0\leq i\leq \hat{N}$ uniformly, suppose $S^2_i$ is the sample variance estimator for $\sigma^2_i$ from $m$ samples taken at $x_i$, then for any $0<s<1$, 
	\[\mathbb{P}\left(\left|\frac{S_i^2}{\sum_{j=0}^{\hat{N}}S_j^2} - \frac{\sigma_i^2}{\sum_{j=0}^{\hat{N}}\sigma_j^2}\right|>s\frac{\sigma_i^2}{\sum_{j=0}^{\hat{N}}\sigma_j^2}\right)\leq \sum_{j=0}^{\hat{N}} K_j,\]
	
	where 
	
	\[K_j = \max\{
	2\exp(-\frac{s^2\sigma_i^2 m}{32(2+s)^2}(1+\frac{1}{\sqrt{m-1}})^{-2})\text{, }
	2\exp(-\frac{s\sigma_im}{8(2+s)}(1+\frac{1}{\sqrt{m-1}})^{-1})
	\}.\]
	
	%
	
\end{proposition}

\begin{proof}
	Observe that $|S^2_i - \sigma_i|\leq \frac{s}{2+s}\sigma^2_i$ for all $i = 0,\dots, \hat{N}$, implies
	
	\begin{align*}
		\left|\frac{S_i^2}{\sum_{j=0}^{\hat{N}}S_j^2} - \frac{\sigma_i^2}{\sum_{j=0}^{\hat{N}}\sigma_j^2}\right| &\leq \frac{(1+\frac{s}{2+s})\sigma_i^2}{(1-\frac{s}{2+s})\sum_j \sigma^2_j} - \frac{\sigma^2_i}{\sum_j \sigma^2_j}\\
		&= (\frac{2+s+s}{2+s-s}-1) \frac{\sigma^2_i}{\sum_j\sigma^2_j}\\
		&= s\frac{\sigma^2_i}{\sum_j\sigma^2_j}.
	\end{align*}
	
	The first inequality holds by noting that if $x,y \in [-t,t]$, the function $g(x,y) = \left|\frac{1+x}{1+y} - 1\right|$  achieves its maximum at $x = t, y = -t$. Thus, by the contrapositive,  $\left|\frac{S_i^2}{\sum_{j=0}^{\hat{N}}S_j^2} - \frac{\sigma_i^2}{\sum_{j=0}^{\hat{N}}\sigma_j^2}\right| > s\frac{\sigma_i^2}{\sum_{j=0}^{\hat{N}}\sigma_j^2}$ implies at least of one of $|S^2_j - \sigma_j^2| > \frac{2}{2+s}\sigma_j^2$. In terms of probability, this implies
	
	\begin{align*}
		\mathbb{P}\left(\left|\frac{S_i^2}{\sum_{j=0}^{\hat{N}}S_j^2} - \frac{\sigma_i^2}{\sum_{j=0}^{\hat{N}}\sigma_j^2}\right|>s\frac{\sigma_i^2}{\sum_{j=0}^{\hat{N}}\sigma_j^2}\right)&\leq \mathbb{P}\left(|S^2_j - \sigma^2_j|> \frac{s}{2+s}\sigma^2_j\text{ for some } j\right) \\
		&\leq \sum_{j=0}^{\hat{N}} \mathbb{P}\left(|S^2_j - \sigma^2_j|> \frac{s}{2+s}\sigma^2_j\right)\\
		& = \sum_{j=0}^{\hat{N}} K_j\\
	\end{align*}
	using the union bound. Now by Lemma \ref{lem:sample_var_bound}, we obtain
	
	\begin{align*}
		K_j &= \max\{2\exp(-\frac{(\frac{s}{2+s}\sigma^2_j)^2}{2\nu^2}),2\exp(-\frac{\frac{s}{2+s}\sigma^2_j}{2\alpha}) \}\\
		&= \max\{2\exp(-\frac{(\frac{s}{2+s}\sigma^2_j)^2}{2(\frac{4\sigma_j}{\sqrt{m}}(1+\frac{1}{\sqrt{m-1}}))^2}), 2\exp(-\frac{\frac{s}{2+s}\sigma^2_j}{2 \frac{4\sigma_j}{m}(1+\frac{1}{\sqrt{m-1}})})\}\\
		&= \max\{2\exp(-\frac{s^2\sigma_j^2m}{32(2+s)^2}(1+\frac{1}{\sqrt{m-1}})^{-2}) , 2\exp(\frac{s\sigma_jm}{8(2+s)}(1+\frac{1}{\sqrt{m-1}})^{-1})\}. 
	\end{align*}
	%
	
\end{proof}

\begin{remark}
	We do not expect the first bound derived in this corollary to be holding tight, because in general the probability of at least one of $|S_i^2 - \sigma_i^2| > \frac{s}{2+s} \sigma_i^2$ can be a considerable overestimate of the probability of $\left|\frac{S_i^2}{\sum_{j=0}^{\hat{N}}S_j^2} - \frac{\sigma_i^2}{\sum_{j=0}^{\hat{N}}\sigma_j^2}\right|>s\frac{\sigma_i^2}{\sum_{j=0}^{\hat{N}}\sigma_j^2}$, and the union bound estimation using $\sum K_j$ is crude as well. However, this bound is good enough as each $K_i$ still decays exponentially with respect to $m$. As long as we have good control of $\hat{N}$, the right hand side will still decay exponentially as $m$ tends to infinity. 
\end{remark}

If we set $N_1 = rN$ for some $0<r<1$, then $m = r\frac{N}{\hat{N}}$. Now if $\hat{N} \ll \sqrt{N}$ which we assume in our implementations, then $ \hat{N} \ll m$, and so $\sum_{j=0}^{\hat{N}} K_j$ as in Proposition \ref{cor:var_proportional_bound} will converge to 0 exponentially. This means for each $0<s$, $ \left|\frac{S_i^2}{\sum_{j=0}^{\hat{N}}S_i^2} - \frac{\sigma_i^2}{\sum_{j=0}^{\hat{N}}\sigma_i^2}\right| \leq s\frac{\sigma_i^2}{\sum_{j=0}^{\hat{N}}\sigma_i^2}$ holds with high probability for large $m$ hence large $N$. Assuming this holds, at each $\sigma_i$, we will have sampled at least $ m+(N-N_1)\frac{S^2_i}{\sum_{j=0} S^2_j} \geq N(1-r)(1-s)\frac{\sigma_i^2}{\sum_{j=0}^{\hat{N}}\sigma_i^2}$ times\footnote{Here we ignore the $m$ pre-samples just to simplify our argument and better demonstrate our conclusion.}. This reduces the standard deviation of noise at $x_i$ from $\sigma_i$ to at most

\[\frac{\sigma_i}{\sqrt{N(1-r)(1-s)\frac{\sigma_i^2}{\sum_{j=0}^{\hat{N}}\sigma_i^2}}} = \frac{1}{\sqrt{(1-s)(1-r)}}\frac{\|\bm{\sigma}\|_2}{\sqrt{N}}. \]

Recall that the noise sampled in Algorithm \ref{alg:noisychebtrunchetero} at each Chebyshev point is roughly $\frac{\|\bm{\sigma}\|_2}{\sqrt{N}}$. Hence if for fixed $s, r$, for large $N$ we find that the noise level sampled at each $x_i$ using HeteroChebtrunc is greater than that of Algorithm \ref{alg:noisychebtrunchetero} by only a constant factor $\frac{1}{\sqrt{(1-r)(1-s)}}$. $ s $ can then be made arbitrarily small given large enough $N$. Therefore, we expect that the approximant $p_n$ of HeteroChebtrunc should achieve a smaller uniform error $\|f-p_n\|$ than the original NoisyChebtrunc, and approaches $\frac{1}{\sqrt{1-r}}$ times the error of Algorithm \ref{alg:noisychebtrunchetero} as $N$ grows large. To be precise, we have proven the following:

\begin{theorem}
    (Uniform Error bound of HeteroChebtrunc) For any $ 0 < s < 1$ fixed, let $f$ be a noisy function on $[-1,1]$ where the noise the heteroskedastic, as in Definition \ref{def:hetero}. Let $N+1$ be the total number of sample budgets, and $p_n$ be the polynomial approximant produced by Algorithm \ref{alg:noisychebtrununknown} with input $f$ and $N+1$, and sampling is carried out at the $\hat{N}+1$ Chebyshev points $\{x_i\}_{i=0}^{\hat{N}}$, with proportion of pre-sampling $0<r<1$. If $\epsilon_{x_i}$ is subgaussian with parameter $ \sigma(x_i)=: \sigma_i $, then for large enough $N$, for any fixed $x\in [-1,1]$,
	\[\mathbb{P} \left[|p_n(x)-f(x)| > 2t\frac{\|\bm{\sigma}\|_2}{\sqrt{N\hat{N}(1-s)(1-r)}} \sqrt{n+1} + (\sqrt{8(n+1)}+1)\|r_n \|_\infty\right] \leq 2\exp(-\frac{t^2}{2}).\] 
\end{theorem}

We will illustrate this behaviour later in the numerical experiment section.

To summarise, we first derive Lemma \ref{lem:sample_var_bound}, a non-asymptotic error bound of the unbiased sample variance estimator $S^2_i$ when sample random variables are i.i.d. subgaussian. We are then able to determine the relative error of our pre-sample estimate $ \frac{S_i^2}{\sum_{j=0}^{\hat{N}}S_i^2}$, which is used to determine the weightings at each $x_i$. Finally, we combine the above results and prove that HeteroChebtrunc can achieve a noise redistribution that is arbitrarily close to that of Algorithm \ref{alg:noisychebtrunchetero} with high probability (converging at exponential rate), and eventually leads to essentially the uniform error rate $O(\frac{\|\bm{\sigma}\|_2}{\sqrt{1-r}}\sqrt{\frac{n}{N}})$ without needing any information on the noise level $\sigma(x)$.

\subsection{Numerical Experiments}\label{sec:num-unknown}

In this section, we demonstrate the properties of HeteroChebtrunc via numerical experiments. We choose to display these results separately from Algorithm \ref{alg:noisychebtrunchetero} to highlight different individual properties of these algorithms. 

\subsubsection{Uniform Error Convergence}

We first demonstrate the conclusion on uniform error rate in Section \ref{sec:sample_var} in Figure \ref{fig:threewaycomparison}. We run NoisyChebtrunc, HeteroChebtrunc and Algorithm \ref{alg:noisychebtrunchetero} with the Runge function $f(x) = \frac{1}{1+25x^2}$. We simulate normal noise with three different noise function: $\sigma_1(x) = |\sin(3x) + 0.00001|$, $\sigma_2(x) = \begin{cases}
	1 & \text{if } x \in [0,1]\\
	0.00001 & \text{otherwise} 
\end{cases}$, and $\sigma_3(x) = \begin{cases}
	10 & \text{if } x \in [0.9,1]\\
	0.00001 & \text{otherwise}
\end{cases}$. We use the choice $r=0.1$ and $\hat{N} = \sqrt{N}$.

\begin{figure}[H]
	\centering
	\centerline{\includegraphics[width=1.1\linewidth]{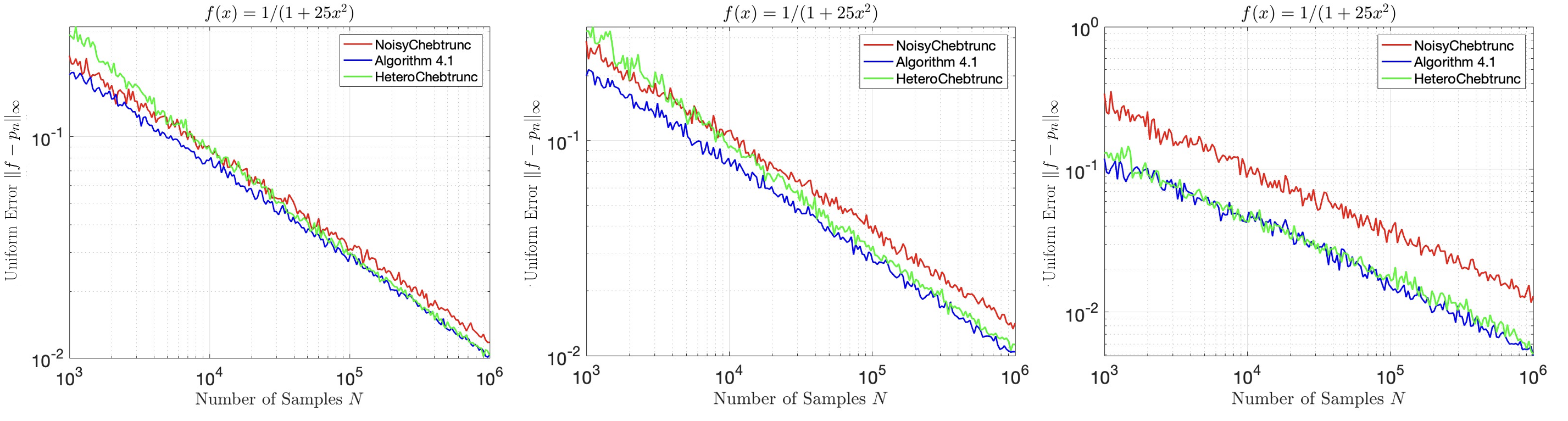}}
	\caption{Uniform error comparison of the three Algorithms. Left: $\sigma_1(x) = \sin(3x) + 0.00001$. Middle: $\sigma_2(x) = \mathbbm{1}_{[0,1]} + 0.00001\mathbbm{1}_{[-1,0)}$. Right: $ \sigma(x) = 10\mathbbm{1}_{[0.9,1]} + 0.00001\mathbbm{1}_{[-1,0.9]} $. Each data point is averaged from 50 trials. In all three experiments HeteroChebtrunc has uniform error close to Algorithm \ref{alg:noisychebtrunchetero} for $N$ large. }
	\label{fig:threewaycomparison}
\end{figure}

The three graphs demonstrate the influence different $\sigma(x)$ has on the uniform error. When $\|\bm{\sigma}\|_\infty$ is close to $\frac{\|\bm{\sigma}\|_2}{\sqrt{\hat{N}}}$ as in $\sigma_1$, and noise redistribution does not reduce maximum noise significantly the three algorithms do not have noticeable differences for large $N$. If $\|\bm{\sigma}\|_\infty$ is considerably larger than $\frac{\|\bm{\sigma}\|_2}{\sqrt{\hat{N}}}$ as in $\sigma_2$, then HeteroChebtrunc performs similarly to NoisyChebtrunc when $N$ is small, and as $N$ grows the pre-sampling improves and its uniform error becomes closer to Algorithm \ref{alg:noisychebtrunchetero}, which is predicted by the analysis after Proposition \ref{cor:var_proportional_bound}. Finally, when the noise is very large in some region, but overall very small on average (as in $\sigma_3(x)$), then both extensions are able to significantly improve the uniform error by taking weighted samples and redistributing noise. We note that in all three cases, the pre-sampling procedure is able to provide sufficient accuracy so that HeteroChebtrunc achieves a uniform error similar to Algorithm \ref{alg:noisychebtrunchetero}.

\begin{figure}[H]
	\centering
	\centerline{\includegraphics[width=0.7\linewidth]{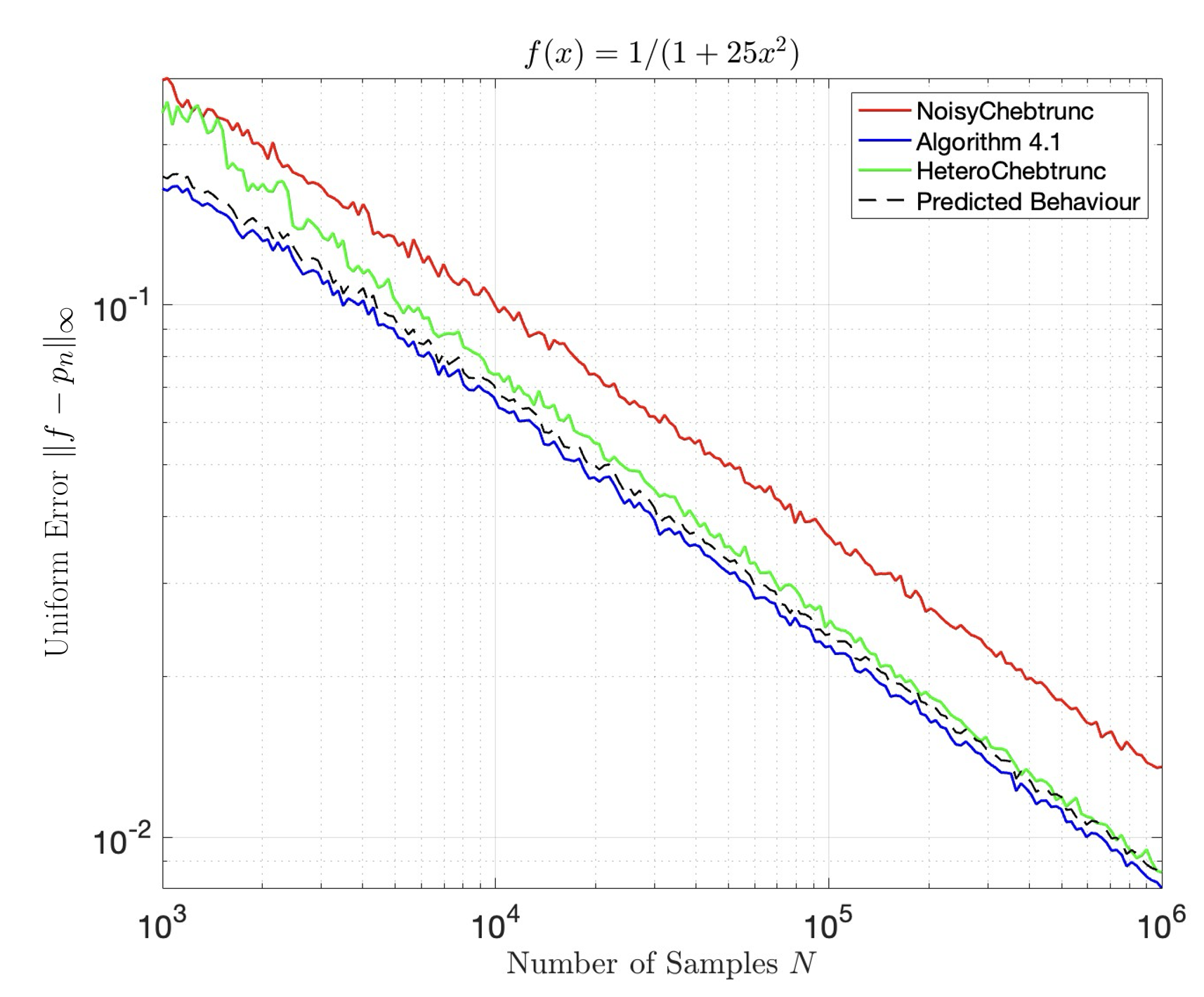}}
	\caption{Uniform error comparison with a large trial size of 500. This plot shows the uniform error $\|f-p_n\|_\infty$ of the three algorithms as $N$ varies from $10^3$ to $10^6$. The black dashed line represents $\frac{1}{\sqrt{1-r}}\|f-p_n\|$, where $p_n$ is computed using Algorithm \ref{alg:noisychebtrunchetero}, the error of HeteroChebtrunc predicted by Proposition \ref{cor:var_proportional_bound}. Each data point is averaged from 500 trials to reduce oscillations. This plot shows HeteroChebtrunc is at least as good as NoisyChebtrunc, and approaches Algorithm \ref{alg:noisychebtrunchetero} for large $N$.} 
	\label{fig:500trials}
\end{figure}

We highlight the property of HeteroChebtrunc that it has uniform error similar to NoisyChebtrunc for small $N$, and approaches the uniform error of Algorithm \ref{alg:noisychebtrunchetero} as $N$ grows and pre-sampling size $m$ increases. To this end, we run the experiment again with the Runge function and noise level $\sigma_4(x) = \begin{cases}
	10 & \text{if } x \in [0,1]\\
	0.00001 & \text{otherwise} 
\end{cases}$, and present our result in Figure \ref{fig:500trials}.

To reduce the oscillations in $\|f-p_n\|_\infty$ observed in previous experiments, each data point was averaged from 500 trials. In this example, the uniform error of NoisyChebtrunc (red) is close to that of NoisyChebtrunc when $N$ is close to $10^4$, but as $N$ increases HeteroChebtrunc has its uniform error gradually approaching Algorithm \ref{alg:noisychebtrunchetero}. We are also able to confirm that near $N=10^6$ the uniform error of HeteroChebtrunc is roughly $\frac{1}{\sqrt{1-r}}$ times the error of Algorithm \ref{alg:noisychebtrunchetero} (shown as black dashed), which is predicted by the discussions following Proposition \ref{cor:var_proportional_bound}. 

\subsubsection{Effect of Heteroskedasticity}

We illustrate how the heteroskedastic noise level $\sigma(x)$ affects the error $f-p_n$, in Figure \ref{fig:f-pn}. The graph plots $f-p_n$ against $x$ when using NoisyChebtrunc (red) and HeteroChebtrunc (green) to compute an approximant of the Runge function with $N=10^7$ and two different noise function $\sigma(x) = \begin{cases}
	1 & \text{if } x \in [0,1]\\
	0.00001 & \text{otherwise} 
\end{cases}$ and $\sigma(x) = \begin{cases}
	1 & \text{if } x \in [0.9,1]\\
	0.00001 & \text{otherwise}
\end{cases}$. 

\begin{figure}[H]
	\centering
	\centerline{\includegraphics[width=1.1\linewidth]{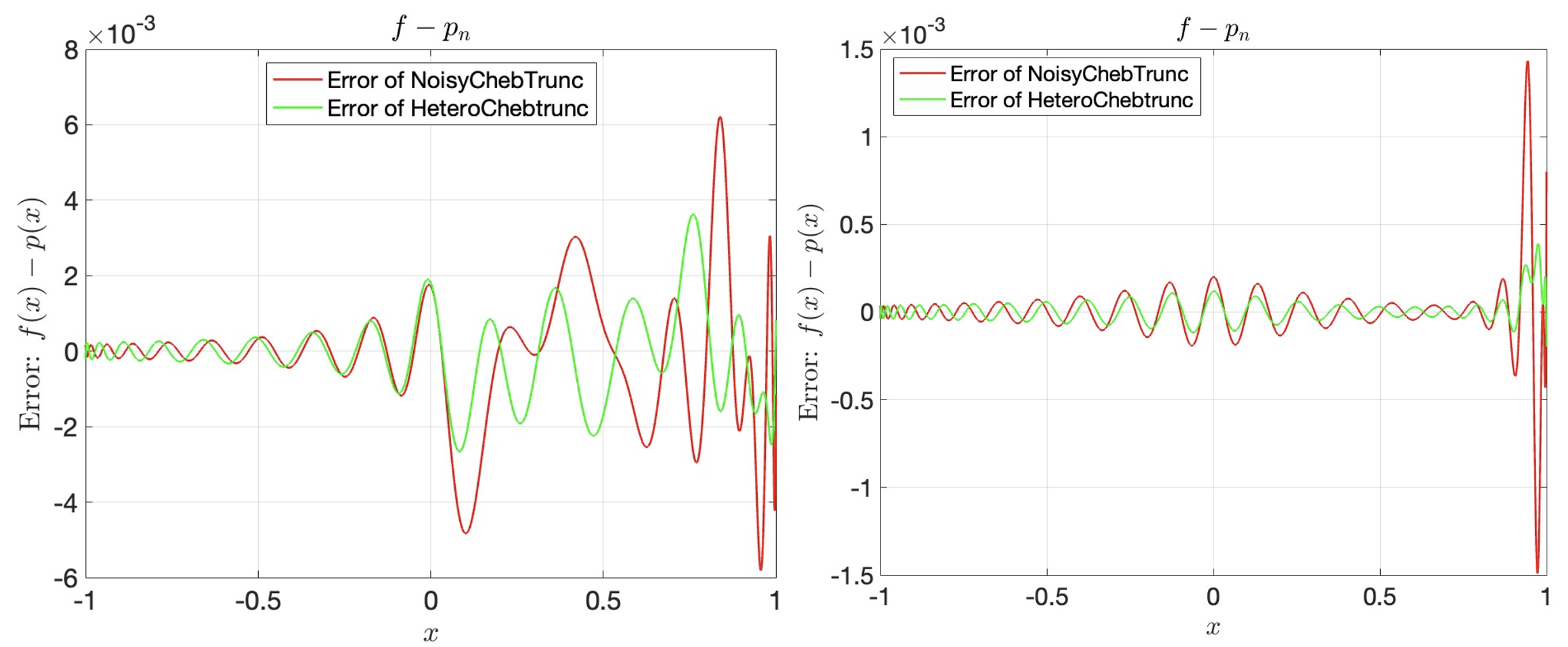}}
	\caption{Graph of $f-p_n$. Left: $\sigma(x) = \mathbbm{1}_{[0,1]} + 0.00001\mathbbm{1}_{[-1,0)}$. Right: $\sigma(x) = 10\mathbbm{1}_{[0.9,1]} + 0.00001\mathbbm{1}_{[-1,0.9]}$. Each plot is randomly selected from 100 trials. HeteroChebtrunc has error comparably more uniform and less influenced by the noise distribution than NoisyChebtrunc, thanks to the accurate pre-sampling and effective noise redistribution.}
	\label{fig:f-pn}
\end{figure}

In the first case (left), the noise is concentrated on $[0,1]$, and we see that the original NoisyChebtrunc has much higher error at this interval than on $[-1,0]$. However, the pre-sampling procedure of HeteroChebtrunc was able to detect the heteroskedasticity and take more samples on $[0,1]$, allowing it to have a more uniform error distribution and lower maximum error. This is perhaps better demonstrated in the burst noise case (right), where the high noise on $[0.9,1]$ caused the NoisyChebtrunc approximant to experience a large error near this region, but HeteroChebtrunc is able to avoid this issue, though an increase in error is still noticeable near $[0.9,1]$. This is because the decay of the noise effect $\sigma_i$ at each $x_i$ is at a rate\footnote{See earlier analysis in Section \ref{sec:noisy-hetero}.} $\frac{1}{\sqrt{k_i}}$, so it is rather slow and requires very large $k_i$ to converge given the high noise there.

\subsubsection{Sample Allocations of HeteroChebtrunc}

We next examine the number of samples allocated to each of $x_i$ in Step 3 of HeteroChebtrunc. Ideally, one would hope that the number of weighted samples taken in HeteroChebtrunc is as close to as if the noise level is known i.e. $k_i = N\frac{\sigma_i^2}{\sum \sigma_j^2}$ as in Algorithm \ref{alg:noisychebtrunchetero}. For those $x_i$ where the noise is small, and $k_i < m$, this is not possible as we must take $m$ pre-samples at every $x_i$. We therefore focus on the case where $k_i \geq m$.

Figure \ref{fig:unknown-allocation} demonstrates the sample allocation and noise variance of HeteroChebtrunc compared with Algorithm \ref{alg:noisychebtrunchetero}. The samples are taken with noise function $\sigma(x) = \begin{cases}
	1 & \text{if } x \in [0,1]\\
	0.00001 & \text{otherwise} 
\end{cases}$. The left plot shows $(x_i, k_i)$, i.e. the number of samples allocated to each Chebyshev points in total, including pre-samples. We use $N = 10^6$, $r=0.1$, and $\hat{N} = 10^3$, so $m=100$. We look at both normally (left) and uniformly distributed\footnote{as in noise follows the uniform distribution $U[-1,1]$} noise (right).

\begin{figure}[H]
	\centering
	\centerline{\includegraphics[width=1.1\linewidth]{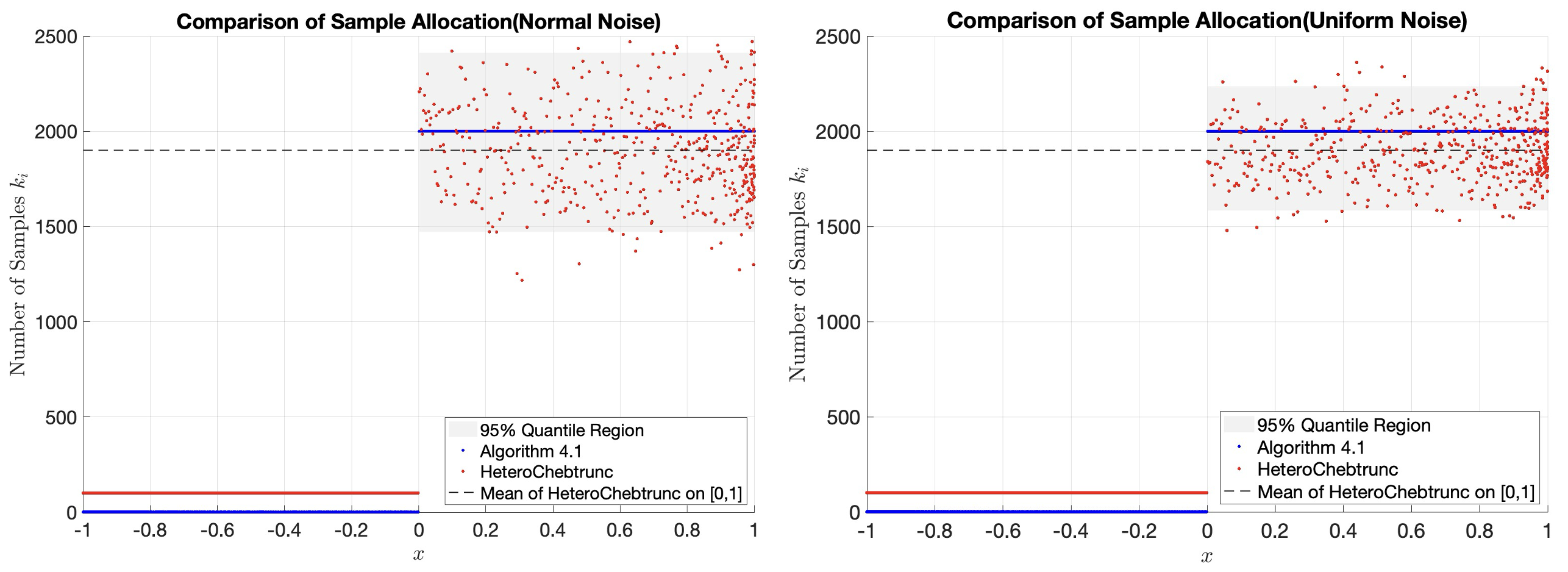}}
	\caption{Number of Samples $k_i$ allocated at each Chebyshev point $x_i$ using HeteroChebtrunc (red) and Algorithm \ref{alg:noisychebtrunchetero} (blue). Mean number of samples on $[0,1]$ is shown as dashed line, regions between 2.5\% and 97.5\% quantile are shaded. Left: normal noise; Right: uniform noise. This experiment illustrates HeteroChebtrunc approximates optimal sampling.}
	\label{fig:unknown-allocation}
\end{figure}

In HeteroChebtrunc (red)\footnote{We use red instead of green here for better visibility}, each $x_i$ receives at least $m$ samples from the pre-sampling. This is slightly inefficient as ideally, less than $m$ samples should be allocated at $x_i$ when noise is very small (e.g. $x_i\in [-1,0]$ in the plot). As a result, on noisier nodes ($x_i\in [0,1]$), the average number of samples (dashed line) on this section (note noise is uniform on this interval) is slightly lower than the optimal choice made by Algorithm \ref{alg:noisychebtrunchetero} (blue), but in general HeteroChebtrunc selects a reasonable number of pre-samples at each $x_i$, and the selection will improve as $N$ increases.

Recall Proposition \ref{cor:var_proportional_bound}, which states that for any $s$, our estimate $\frac{S_i^2}{\sum_{j=0}^{\hat{N}}S_j^2} $ lie in $ [(1- s)\frac{\sigma_i^2}{\sum_{j=0}^{\hat{N}}\sigma_j^2}, (1+ s)\frac{\sigma_i^2}{\sum_{j=0}^{\hat{N}}\sigma_j^2}]$ with probability\footnote{We ignore the $(1+\frac{1}{\sqrt{m-1}})$ term as it tends to 1.} $ 1-A_{\hat{N}}\exp(C_{s,\sigma_i}m)$ for some constant $A_{\hat{N}}$ depending on $\hat{N}$, and $C_{s,\sigma_i}$ depending on $s$ and $\sigma_i$. In order to simplify theoretical analysis, we applied some loose bounds in the proof, hence the bound there is not sharp enough to serve as an indicator for the empirical errors. Also, our bound is very general for any form of subgaussian noise, and if noise is for example normal one might be able to devise a sharper bound using other statistical techniques. This suggests that one may expect faster convergence rate from the sample variance estimator than Proposition \ref{cor:var_proportional_bound} suggests.

\subsubsection{Runtime Comparison}

To end this section, we compare the runtime of NoisyChebtrunc and HeteroChebtrunc. Recall the time complexity of HeteroChebtrunc is $O(N+\hat{N}\log \hat{N})$, which should be moderately faster than NoisyChebtrunc's $O(N\log N)$. We confirm this with our numerical experiment: 
\begin{figure}[h]
    \centering
    \includegraphics[width=0.75\linewidth]{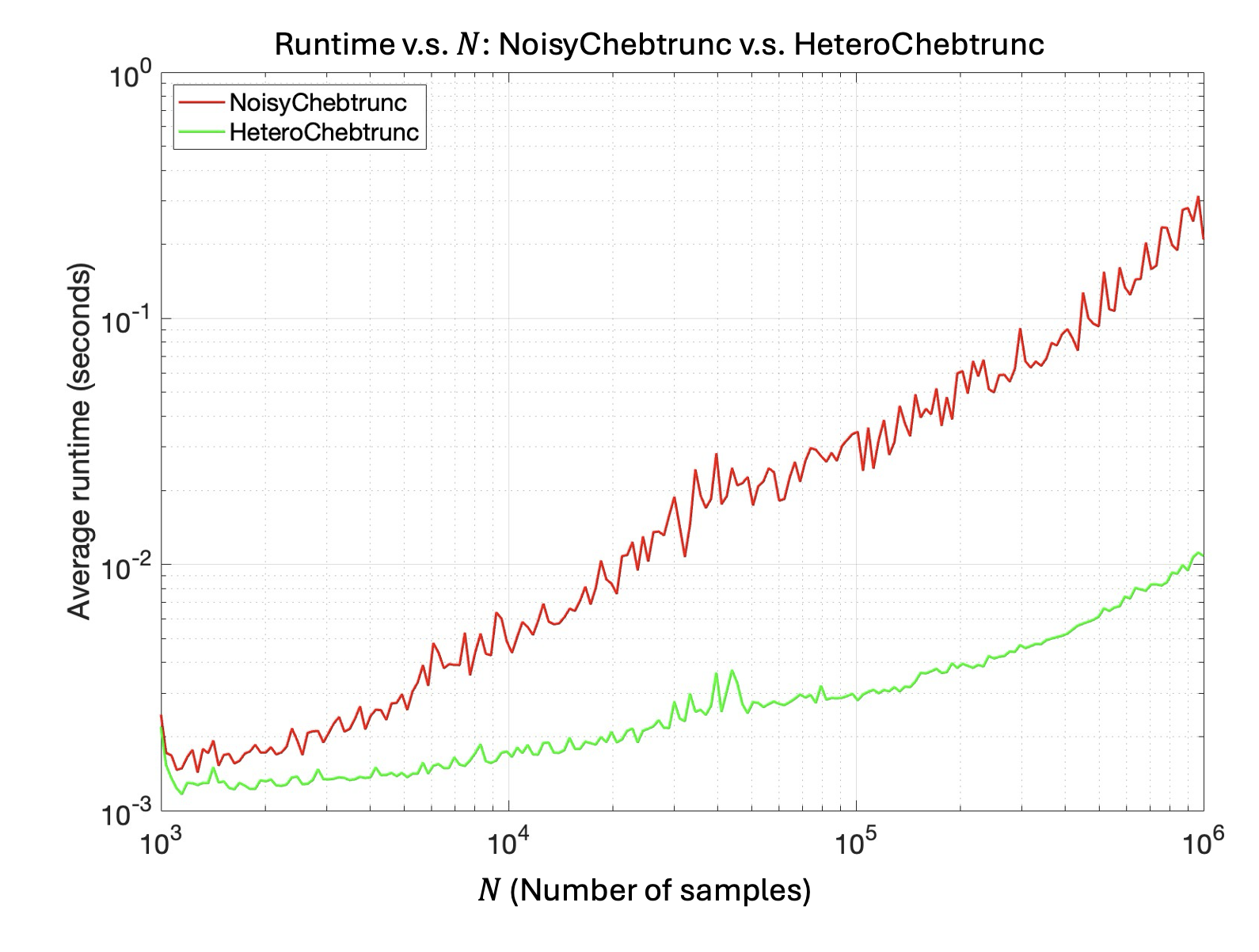}
    \caption{Runtime of NoisyChebtrunc (Red) and HeteroChebtrunc (Green) with $N $ varying from $10^3$ to $10^6$. Each data point is averaged from 50 trials. $\hat{N}$ is chosen to be $\sqrt{N} $ and $r = 0.1$. HeteroChebtrunc runs significantly faster than NoisyChebtrunc empirically, thanks to its $O(N)$ time complexity.}
    \label{fig:placeholder}. 
\end{figure}
We run NoisyChebtrunc and HeteroChebtrunc over 200 logarithmically spaced points from $10^3$ to $10^6$, and the sampling procedure used is the MATLAB \texttt{randn} function. We see significant improvement in runtime for HeteroChebtrunc comparing with NoisyChebtrunc, and the improvement increases as $N$ grows, which is expected from the difference in time complexity. 

It should be noted that in applications where sampling is expensive, e.g. one sample is a numerical evaluation of a PDE, then the difference in runtime might not be as significant. 

\section{Dependent Noise}

Given $x,y \in [-1,1]$, $x\neq y$, if the noise random variables $\epsilon_x$ and $\epsilon_y$ are dependent, then the analysis becomes much more difficult. For example, consider the extreme scenario where $\epsilon_x = \epsilon_y$ for all $x,y$, then the Chebyshev interpolant from these noisy data would simply be a vertical shift of the interpolant of $f$. Therefore, if one applies the original NoisyChebtrunc, the same noise reduction effect cannot hold. However, if we make the additional assumption that one can repeatedly collect independent samples at the same input $x_i$, which is required by HeteroChebtrunc, then a noise reduction effect can be achieved. To reiterate our assumptions, we assume that at each $x_i$, we can collect multiple independent samples from the random variable $\epsilon_{i} := \epsilon_{x_i}$, but $\epsilon_{x_i}$ and $\epsilon_{x_j}$ might not be independent. 

\begin{theorem}
	Let $\hat{p}_{\hat{N}}$ be the noisy interpolant through $\{(x_i, y_i)\}_{i=0}^{\hat{N}}$, where $y_i =f(x_i) + \epsilon_i$. Assuming that one can take multiple independent samples from each of the random variable $\epsilon_{i}$, then \[\mathbb{P}\left(\|f- \hat{p}_{\hat{N}} \|_\infty \geq \|q_{\hat{N}}\|_\infty  + t\frac{\|\bm{\sigma}\|_2}{\sqrt{\hat{N}}}\sqrt{\frac{{\hat{N}}}{N}}(\frac{2}{\pi} \log (\hat{N}+1)+1) \right) \leq  2\hat{N} \exp(-t^2).\]
	
\end{theorem}

We first claim that for large $N$, the noises $\epsilon_i$ is bounded uniformly with high probability:

\begin{lemma}
	For $i =0, 1,\dots, \hat{N}$, $\epsilon_i$ subgaussian, without assuming any form of independence, 
	
	\[\mathbb{P}(|\epsilon_i| > \frac{\|\bm{\sigma}\|_2}{\sqrt{N}} t \text{ for some }i) \leq 2\hat{N} \exp(-t^2) .\]
\end{lemma}

\begin{proof}
	This is directly from the definition of subgaussian random variable:
	
	\begin{align*}
		\mathbb{P}(|\epsilon_i| > \frac{\|\bm{\sigma}\|_2}{\sqrt{N}} t \text{ for some }i) & = \mathbb{P}(\bigcup_{i=0}^{\hat{N}}|\epsilon_i| > \frac{\|\bm{\sigma}\|_2}{\sqrt{N}} t )\\
		&\leq \sum_{i=0}^{\hat{N}} \mathbb{P}(|\epsilon_i| > \frac{\|\bm{\sigma}\|_2}{\sqrt{N}} t )\\
		&\leq 2\hat{N} \exp(-t^2).
	\end{align*}
\end{proof}

Now let $q_n = f - \bar{p}_n$ be the error of the exact Chebyshev interpolant $\bar{p}_n$ of $n+1$ points, so $\|q_n\|_\infty $ achieves spectral convergence to 0. We are going to find the error of our noisy interpolant using the error of exact interpolant. The key idea is that if all of $\epsilon_i$ are bounded, then the well-conditionedness of Chebyshev interpolation implies that the noisy interpolant $\hat{p}_n$ cannot deviate from $\bar{p}_n$ too much.

\begin{proof}(of theorem)
	Let $\rho = \frac{\|\bm{\sigma}\|_2}{\sqrt{N}}$. If each of $|\epsilon_i| \leq \rho t$, then we note $\bar{p}_{\hat{N}} - \hat{p}_{\hat{N}}$ is a Chebyshev interpolant through the points $(x_i, \epsilon_i)$, so the Lebesgue constant $\Lambda_{\hat{N}}$ of Chebyshev interpolation gives
	
	\[\|\bar{p}_{\hat{N}} - \hat{p}_{\hat{N}}\| \leq \rho t(\frac{2}{\pi} \log (\hat{N}+1) +1). \]
	
	Hence $\|f-\hat{p}_{\hat{N}} \| \leq \|q_{\hat{N}}\| + \|\bar{p}_{\hat{N}} - \hat{p}_{\hat{N}}\| \leq \|q_{\hat{N}}\|+ \rho t(\frac{2}{\pi} \log (\hat{N}+1) +1).$
	
	Therefore, if $\|f-\hat{p}_{\hat{N}}\|_\infty $ is greater than RHS above, then at least one of $|\epsilon_i| > \rho t$, the result follows from the previous lemma.
	
\end{proof}

Although our theorem concerns the Chebyshev interpolant $\hat{p}_{\hat{N}}$, not the approximant after truncation using Mallow's $C_p$, we expect the final approximant from HeteroChebtrunc to satisfy a similar error bound.

\section{Discussions}

HeteroChebtrunc improves the accuracy and time complexity of NoisyChebtrunc when the noises are no longer identical, under the additional assumption that one can collect independent samples at the same $x$, We also showed that HeteroChebtrunc achieves a similar uniform error bound if the noise random variables are dependent, but the aforementioned assumption makes this result less practical. Whether the same results on both accuracy and dependence can be achieved without this assumption is a natural extension of this work.

We also have yet to find a lower bound for the choice of $\hat{N}$. Specifically, such a bound might depend on the function $f$ as Chebyshev interpolation converges at different rates for different $f$. Our experiments use $\hat{N} = \lfloor \sqrt{N}\rfloor$, which is large enough for most implementations, and we have seen in Section \ref{sec:nhat-choice} that $\hat{N}$ does not substantially impact uniform error as long as it is moderately large. But ideally, one would like to choose $\hat{N}$ to be as small as possible, because in HeteroChebtrunc a smaller $\hat{N}$ allows more samples $m = r\frac{N}{\hat{N}}$ to be allocated to each of the $x_i$, improving the accuracy of the pre-sampling and hence the overall procedure. Hence, a sharp lower bound for $\hat{N}$ would be useful for the implementation of HeteroChebtrunc. 

We have mainly worked on the case where the sample points can be chosen to be Chebyshev. Another important direction would be to generalise this to cases where $x_i$ cannot be freely selected, such as equispaced points.

Heteroskedastic noise is prevalent in many applications. For example, in Bayesian optimisation, black-box approximation of a noisy function is an important but often difficult task, particularly if noise is heteroskedastic \cite{black-box, black-box-2,black-box-3}. Another example is derivative estimation as mentioned already in \cite{yuji}. An adaptation and analysis of HeteroChebtrunc in these fields are left for future work.

\bmhead{Acknowledgements}

We thank Rebecca Lewis for the valuable conversations regarding concentration inequalities.

\bibliography{main}

\end{document}